\documentclass[10pt]{article}
\usepackage{amsmath}
\usepackage{amssymb}
\usepackage{amsthm}
\usepackage{mathrsfs}
\usepackage{tikz,authblk}
\newtheorem{theorem}{Theorem}[section]
\newtheorem{lemma}[theorem]{Lemma}
\newtheorem{corollary}[theorem]{Corollary}

\newtheorem{problem}[theorem]{Problem}
\newtheorem{observation}[theorem]{Observation}
\theoremstyle{definition}
\newtheorem{definition}[theorem]{Definition}

\newtheorem{remark}[theorem]{Remark}

\begin{document}
\title{On kernels by rainbow paths in arc-coloured digraphs\thanks{Corresponding author.
$E$-$mail\ \ address:$ ruijuanli@sxu.edu.cn (R. Li). Research of RL is partially supported by NNSFC under no. 11401353 and TYAL of Shanxi. }}
\author{Ruijuan Li}
\author{Yanqin Cao}
\affil{School of Mathematical Sciences, Shanxi University, Taiyuan, Shanxi, 030006, PR China}

\maketitle

\begin{abstract}
In 2018, Bai, Fujita and Zhang (\emph{Discrete Math.} 2018, 341(6): 1523-1533) introduced the concept of a kernel by rainbow paths (for short, RP-kernel) of an arc-coloured digraph $D$, which is a subset $S$ of vertices of $D$ such that ($a$) there exists no rainbow path for any pair of distinct vertices of $S$, and ($b$) every vertex outside $S$ can reach $S$ by a rainbow path in $D$. They showed that it is NP-hard to recognize wether an arc-coloured digraph has a RP-kernel and it is NP-complete to decided wether an arc-coloured tournament has a RP-kernel.
In this paper, we give the sufficient conditions for the existence of a RP-kernel in arc-coloured unicyclic digraphs, semicomplete digraphs, quasi-transitive digraphs and bipartite tournaments, and prove that these arc-coloured digraphs have RP-kernels if certain ``short" cycles and certain ``small" induced subdigraphs are rainbow.
\end{abstract}

\noindent{\bf Keywords:} arc-coloured digraphs; kernels; kernels by rainbow paths
\vskip 3mm

\noindent{\bf MR(2010) Subject Classification:} 05C20, 05C12, 05C07 

\section{Introduction}

For convenience of the reader, some necessary terminology and notation not mentioned in this section can be found in Section $2$. All digraphs considered in this paper are finite. In this paper, all paths, walks and cycles are always directed. For terminology and notation, we refer the reader to Bang-Jensen and Gutin \cite{1}.

Let $D$ be a digraph. A \emph{kernel} of $D$ is a subset $S\subseteq V(D)$ such that ($a$) for any pair of distinct vertices $x,y\in S$ are non-adjacent, and ($b$) for each vertex $v\in V(D)\setminus S$, there exists a vertex $s\in S$ such that $(v,s)\in A(D)$. This notion was originally introduced in the game theory by von Neumann and Morgenstern \cite{3} in 1944. Kernels have found many applications and several sufficient conditions for the existence of a kernel have been proved. In this paper, we will need the following result.
\begin{theorem}\cite{3}\label{z1.1} Let $D$ be a digraph. If $D$ has no cycle, then $D$ has a unique kernel. \end{theorem}

Let $D$ be a digraph and $m$ a positive integer. An arc-colouring of $D$ is a mapping $C: A(D)\rightarrow N$, where $N$ is the set of natural numbers. $C(D)$ and $C(x,y)$ denote the set colours appearing on all the arcs of $D$ and the colour appearing on the arc $(x,y)\in A(D)$. We call $D$ an \emph{$m$-arc-coloured digraph} if $D$ has such an arc-colouring with $|C(D)|=m$. An arc-coloured digraph is called \emph{monochromatic} if all arcs are assigned same colour. Define a \emph{kernel by monochromatic paths} of an arc-coloured digraph $D$ to be a subset $S\subseteq V(D)$ such that ($a$) there exists no monochromatic path for any pair of vertices of $S$, and ($b$) for each vertex outside $S$ can reach $S$ by a monochromatic path.

The concept of a kernel by monochromatic paths of an arc-coloured digraph was introduced by Sands, Sauer and Woodrow \cite{2} in 1982 as a generalization of the concept of a kernel. They showed that every $2$-coloured digraph has a kernel by monochromatic paths. As a corollary, they showed that every $2$-coloured tournament has a one-vertex kernel by monochromatic paths. They also proposed the problem that whether a $3$-coloured tournament with no rainbow triangle has a one-vertex kernel by monochromatic paths. In 1988, Shen \cite{4} proved that for $m\geqslant3$ every $m$-coloured tournament with no rainbow triangle and no rainbow transitive triangle has a one-vertex kernel by monochromatic paths, and also showed that the condition ``with no rainbow triangle and no rainbow transitive triangle'' cannot be improved for $m\geqslant5$. In 2004, Galeana-S\'{a}nchez and Rojas-Monroy \cite{6} showed the condition of Shen cannot be improved for $m=4$ by constructing a family of counterexamples. For more results about kernels by monochromatic paths of an arc-coloured digraph can be found in \cite{7,14,8,9,10,12}.

An arc-coloured digraph is called \emph{properly coloured} if any two consecutive arcs have distinct colours. Define a \emph{kernel by properly coloured paths} of an arc-coloured digraph $D$ to be a subset $S\subseteq V(D)$ such that ($a$) there exists no properly coloured path for any pair of vertices of $S$, and ($b$) for each vertex outside $S$ can reach $S$ by a properly coloured path.

The concept of a kernel by properly coloured paths of an arc-coloured digraph was introduced by Delgado-Escalante and Galeana-S\'{a}nchez \cite{15} in 2009 as a generalization of the concept of a kernel. Bai, Fujita and Zhang \cite{11} showed in 2018 that it is NP-hard to recognize wether an arc-coloured digraph has a kernel by properly coloured paths. They conjecture that every arc-coloured digraph with all cycles properly coloured has a kernel by properly coloured paths and verified the conjecture for unicyclic digraphs, semicomplete digraphs and bipartite tournaments. In 2018, Delgado-Escalante, Galeana-S\'{a}nchez and O'Reilly-Regueiro \cite{13} gave some sufficient conditions for the existence of a kernel by properly coloured paths in arc-coloured tournaments, quasi-transitive digraphs and $k$-partite tournaments.

An arc-coloured digraph is called \emph{rainbow} if all arcs have distinct colours. Define a \emph{kernel by rainbow paths} (for short, \emph{RP-kernel}) of an arc-coloured digraph $D$ to be a subset $S\subseteq V(D)$ such that ($a$) there exists no rainbow path for any pair of vertices of $S$, and ($b$) for each vertex outside $S$ can reach $S$ by a rainbow path.

The concept of a RP-kernel of an arc-coloured digraph was introduced by Bai, Fujita and Zhang \cite{11} in 2018 as a generalization of the concept of kernel. They showed that it is NP-hard to recognize wether an arc-coloured digraph has a RP-kernel. Recently, Bai, Li and Zhang \cite{16} showed the following theorem and proposed the following problem.
\begin{theorem}\cite{16}
It is NP-complete to decided wether an arc-coloured tournament has a RP-kernel.
\end{theorem}
\begin{problem}\cite{16}
Is it true that every arc-coloured digraph with all cycles rainbow has a RP-kernel?
\end{problem}
In this paper, we give some sufficient conditions for the existence of a RP-kernel in arc-coloured unicyclic digraphs, semicomplete digraphs, quasi-transitive digraphs and bipartite tournaments and prove that these arc-coloured digraphs have RP-kernels if certain ``short" cycles and certain ``small" induced subdigraphs are rainbow.

\section{Terminology and Preliminaries}

Let $D$ be a digraph. $V(D)$ and $A(D)$ denote its vertex and arc sets. If $(x,y)$ is an arc of $D$, sometimes we use the notation $x\rightarrow y$ to denote this arc. The \emph{out-neighbourhood} (resp. \emph{in-neighbourhood}) of a vertex $x\in V(D)$ is the $N_D^+(x)=\{y\,|\,(x,y)\in A(D)\}$ (resp. $N_D^-(x)=\{y\,|\,(y,x)\in A(D)\}$). For a vertex $x\in V(D)$, the \emph{out-degree} (resp. \emph{in-degree}) of $x$ is denoted by $d_D^+(x)=|N_D^+(x)|$ (resp. $d_D^-(x)=|N_D^-(x)|$). A vertex in $D$ is called \emph{sink} (resp. \emph{source}) if $d_D^+(x)=0$(resp. $d_D^-(x)=0$). An arc $(x,y)\in A(D)$ is called \emph{asymmetrical} (resp. \emph{symmetrical}) if $(y,x)\notin A(D)$ (resp. $(y,x)\in A(D)$). If $S$ is a nonempty set of $V(D)$, then the subdigraph $D[S]$ induced by $S$ is the digraph having vertex set $S$, and whose arcs are all those arcs of $D$ joining vertices of $S$.

For disjoint sets $X$ and $Y$, $X\Rightarrow Y$ means that every vertex of $X$ dominates every vertex of $Y$ and $y\nrightarrow x$ for any $x\in X$ and $y\in Y$. If $Y=\{v\}$, we always denote $X\Rightarrow v$ instead of $X\Rightarrow \{v\}$. $X\nRightarrow Y$ means that there exists a vertex $u\in Y$ such that $X\nRightarrow u$.

For two distinct vertices $x,y\in V(D)$, a path $P$ from $x$ to $y$ is denoted by $(x,y)$-path, $\ell(P)$ denote the length of path $P$. Let $S,K\subseteq V(D)$. $(x,S)$-path in $D$ denote an $(x,s)$-path for some $s\in S$. $(S,x)$-path in $D$ denote an $(s,x)$-path for some $s\in S$. A closed path is called a cycle.
We always call a cycle $C$ of length $\ell(C)$ by $\ell(C)$-cycle.

In the following proof, we use the definition below.
\begin{definition}\label{z1.3}
For an arc-coloured digraph $D$, the \emph{rainbow closure} of $D$ denoted by $C_r(D)$, is a digraph such that:

($a$) $V(C_r(D))=V(D)$;

($b$) $A(C_r(D))=A(D)\cup \{(u,v)\,|\,\mbox{there exists a rainbow }(u,v)\mbox{-path in } D\}$.
\end{definition}

It is not hard to see the following simple and useful result.
\begin{observation}\label{z1.4}
An arc-coloured digraph $D$ has a RP-kernel if and only if $C_r(D)$ has a kernel.
\end{observation}

A digraph $D$ is called a \emph{kernel-perfect digraph} or \emph{KP-digraph} when every induced subdigraph of $D$ has a kernel. The following theorem give a sufficient condition for a digraph to be a \emph{KP}-digraph.
\begin{theorem}\cite{5}\label{z1.2}
Let $D$ be a digraph such that every cycle in $D$ has at least one symmetrical arc. Then $D$ is a KP-digraph.
\end{theorem}

\section{Unicyclic digraphs}
 A digraph $D$ is a \emph{unicyclic digraph} if it contains only one cycle. In this section, we consider the sufficient conditions for the existence of a RP-kernel in an arc-coloured unicyclic digraph.
\begin{theorem}\label{z5.1}
Let $D$ be an $m$-arc-coloured unicyclic digraph such that the unique cycle is rainbow. Then $D$ has a RP-kernel.
\end{theorem}
\begin{proof}
Let $D$ be an $m$-arc-coloured unicyclic digraph with the unique cycle $C$. We will show the result by constructing a RP-kernel $S$ of $D$. If $D$ is strong, then $D=C$. Since the cycle $C$ is rainbow, each vertex of $C$ forms a RP-kernel of $C$. The desired result follows directly.

Now assume $D$ is not strong. Then $D$ has strong components $D_1, D_2 ,\ldots, D_k$ $(k\geqslant2)$ such that there exists no arc from $D_i$ to $D_j$ for any $i>j$. Since $D$ is unicyclic, then one of the strong components $D_1, D_2 ,\ldots, D_k$ containing the unique cycle $C$ and any other strong component is a single vertex. If $D_k$ is a single vertex, say $v=D_k$, we put $v$ into $S$. If $D_k=C$, we put an arbitrary vertex of $C$, say also $v$, into $S$. Since $C$ is rainbow, $V(C)\setminus \{v\}$ can reach $v$ by a rainbow path. Let $j_1\in \{1,2,\ldots,k-1\}$ be the largest integer such that there exists no rainbow path from some vertex of $D_{j_1}$ to $S$. If $D_{j_1}$ is a single vertex, say $v_{j_1}=D_{j_1}$, we put $v_{j_1}$ into $S$. If $D_{j_1}=C$, we put an arbitrary vertex satisfying the condition above of $C$, say also $v_{j_1}$, into $S$. Let $j_2\in \{1,2,\ldots,j_1-1\}$ be the largest integer such that there exists no rainbow path from some vertex of $D_{j_2}$ to $S$. If $D_{j_2}$ is a single vertex, say $v_{j_2}=D_{j_2}$, we put $v_{j_2}$ into $S$. If $D_{j_2}=C$, we put an arbitrary vertex satisfying the condition above of $C$, say also $v_{j_2}$, into $S$. Continue this procedure until all the remaining vertices in $V(D)\backslash S$ can reach $S$ by a rainbow path. Let $v_{j_r}$ be the last vertex putting into $S$. It is not hard to check that the vertex set $S=\{v_{j_1},v_{j_2},\ldots,v_{j_r},v\}$ is a RP-kernel of $D$. \end{proof}

\section{Semicomplete digraphs}

A digraph $D$ is \emph{semicomplete} if for any pair of vertices there exists at least one arc between them. A tournament is a semicomplete digraph with no $2$-cycle. In this section, we consider the sufficient conditions for the existence of a RP-kernel in an arc-coloured semicomplete digraph. Since each pair of vertices in a semicomplete digraph are adjacent, it follows that a RP-kernel of a semicomplete digraph consists of only one vertex.


\begin{theorem}\label{z2.2}
Let $D$ be an $m$-arc-coloured semicomplete digraph with all $3$-cycles are rainbow in $D$. Then $D$ has a one-vertex RP-kernel.
\end{theorem}

\begin{proof}
Let $v$ be a vertex of $D$ with maximum in-degree. Since $D$ is a semicomplete digraph, it follows that $N_D^+(v)\cup N_D^-(v)\cup \{v\}=V(D)$. Moreover, for any $u\in N_D^+(v)$, there exists a vertex $w\in N_D^-(v)$ such that $u\rightarrow w$. If not, then for any $w\in N_D^-(v)$, we have $w\rightarrow u$. This implies that $N_D^-(v)\cup \{v\}\subseteq N_D^-(u)$, which contradicts the choice of $v$. So $(u,w,v,u)$ is a rainbow $3$-cycle. It follows $(u,w,v)$ is a rainbow $(u,v)$-path. Now for any $u\in N_D^+(v)$, there exists a rainbow $(u,v)$-path. Clearly, for any $w\in N_D^-(v)$, there exists a rainbow $(w,v)$-path. Combining with $N_D^+(v)\cup N_D^-(v)\cup \{v\}=V(D)$, we have for any $x\in V(D)\setminus\{v\}$, there exists a rainbow $(x,v)$-path. Thus $\{v\}$ is a one-vertex RP-kernel of $D$.
\end{proof}

\begin{corollary}\cite{16}\label{z2.3}
Let $T$ be an $m$-arc-coloured tournament with all $3$-cycles are rainbow in $D$. Then $T$ has a one-vertex RP-kernel.
\end{corollary}

\section{Quasi-transitive digraphs}
A digraph $D$ is a \emph{quasi-transitive digraph} if whenever $\{(u,v),(v,w)\}\subseteq A(D)$ then either $(u,w)\in A(D)$ or $(w,u)\in A(D)$. In this section, we consider the sufficient conditions for the existence of a RP-kernel in an arc-coloured quasi-transitive digraph.

\begin{lemma}\cite{1}\label{z4.1}
Let $D$ be a quasi-transitive digraph. If for any pair of distinct vertices $x,y$ of $D$ such that $D$ has an $(x,y)$-path but $x$ does not dominate $y$, then either $y\rightarrow x$, or there exist vertices $u,v\in V(D)\setminus\{x,y\}$ such that $x\rightarrow u\rightarrow v\rightarrow y$ and $y\rightarrow u\rightarrow v\rightarrow x$.
\end{lemma}
Let $QT_4$ be a quasi-transitive digraph, which has $V(QT_4)=\{x,y,u,v\}$ and $A(QT_4)=\{(x,u),(u,v),(v,y),(y,u),(v,x)\}$. See Figure 1.

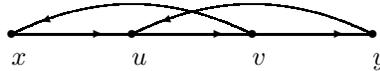
\begin {figure}[h]
\unitlength0.4cm
\begin{center}
\begin{picture}(15,2)
\put(1,1){\circle*{.3}}
\put(5,1){\circle*{.3}}
\put(9,1){\circle*{.3}}
\put(13,1){\circle*{.3}}
\qbezier(1,1)(3,1)(5,1)
\put(4,1){\vector(1,0){.07}}
\qbezier(5,1)(7,1)(9,1)
\put(8,1){\vector(1,0){.07}}
\qbezier(9,1)(11,1)(13,1)
\put(12,1){\vector(1,0){.07}}

\qbezier(13,1)(9,3)(5,1)
\put(6,1.45){\vector(-3,-1){.07}}
\qbezier(9,1)(5,3)(1,1)
\put(2,1.45){\vector(-3,-1){.07}}
\put(1,0){$x$}
\put(5,0){$u$}
\put(9,0){$v$}
\put(13,0){$y$}
\end{picture}
\caption{A quasi-transitive digraph $QT_4$. }
\end{center}
\end{figure}

\begin{lemma}\label{z4.2}
Let $D$ be an $m$-arc-coloured quasi-transitive digraph with all $3$-cycles and all induced subdigraphs $QT_4$ are rainbow in $D$. If for any pair of distinct vertices $x,y\in V(D)$ such that $D$ has a rainbow $(x,y)$-path but no rainbow $(y,x)$-path, then $(x,y)\in A(D)$.
\end{lemma}

\begin{proof}
Suppose to the contrary that $(x,y)\notin A(D)$. Since there exists no rainbow $(y,x)$-path, then $x,y$ are non-adjacent. Let $P=(x=x_0,x_1,x_2,\ldots,x_n=y)\subseteq D$ be a rainbow $(x,y)$-path. By Lemma \ref{z4.1}, there exist vertices $u,v\in V(D)\setminus\{x,y\}$ such that $x\rightarrow u\rightarrow v\rightarrow y$ and $y\rightarrow u\rightarrow v\rightarrow x$. This implies that $D[x,u,v,y]$ is $QT_4$ which is rainbow. It follows $(y,u,v,x)$ is a rainbow $(y,x)$-path, a contradiction. Thus $(x,y)\in A(D)$.
\end{proof}

\begin{theorem}\label{z4.3}
Let $D$ be an $m$-arc-coloured quasi-transitive digraph with all $3$-cycles and all induced subdigraphs $QT_4$ are rainbow in $D$. Then $C_r(D)$ is a KP-digraph.
\end{theorem}
\begin{proof}Suppose to the contrary that $C_r(D)$ is not a \emph{KP}-digraph. By Theorem \ref{z1.2}, there exists a cycle with no symmetrical arc. Let $C=(u_1,u_2,\ldots,u_{\ell},u_1)$ be a shortest cycle with no symmetrical arc in $C_r(D)$. We will get a contradiction by showing that $C$ has a symmetrical arc.

\vskip 2mm
\noindent{\bf Claim 1. }$C\subseteq D$ and $\ell\geqslant 5$.

\begin{proof}Since $C$ has no symmetrical arc, for each $i\in \{1,2,\ldots,\ell\}$, there exists a rainbow $(u_i,u_{i+1})$-path and no rainbow $(u_{i+1},u_i)$-path in $D$. By Lemma \ref{z4.2}, we have $(u_i,u_{i+1})\in A(D)$. Then $C\subseteq D$.

Now we prove $\ell\geqslant 5$. Since $C$ has no symmetrical arc, we have $\ell\geqslant3$.

If $\ell=3$, combining with $C\subseteq D$, we have $C=(u_1,u_2,u_3,u_1)$ is a rainbow $3$-cycle in $D$. This implies that $(u_2,u_3,u_1)$ is a rainbow $(u_2,u_1)$-path and hence $(u_2,u_1)\in A(C_r(D))$. Note that $(u_1,u_2)\in A(C)$, which contradicts that $C$ has no symmetrical arc.

If $\ell=4$, by the proof above, we have $C=(u_1,u_2,u_3,u_4,u_1)$ is a cycle in $D$. Since $D$ is quasi-transitive, we have $u_1,u_3$ are adjacent. If $(u_1,u_3)\in A(D)$, then $(u_1,u_3,u_4,u_1)$ is a rainbow $3$-cycle. This implies that $(u_4,u_1,u_3)$ is a rainbow $(u_4,u_3)$-path and hence $(u_4,u_3)\in A(C_r(D))$. Note that $(u_3,u_4)\in A(C)$, which contradicts that $C$ has no symmetrical arc. If $(u_3,u_1)\in A(D)$, then $(u_1,u_2,u_3,u_1)$ is a rainbow $3$-cycle. This implies that $(u_2,u_3,u_1)$ is a rainbow $(u_2,u_1)$-path and hence $(u_2,u_1)\in A(C_r(D))$. Note that $(u_1,u_2)\in A(C)$, which contradicts that $C$ has no symmetrical arc.

Thus $\ell\geqslant5$.\end{proof}

By Claim 1, we have $\ell-1\geqslant4$ and $C\subseteq D$. Considering $\{(u_{\ell-1},u_{\ell}),(u_{\ell},u_1)\}\subseteq A(D)$, we have $u_1,u_{\ell-1}$ are adjacent.

If $(u_1,u_{\ell-1})\in A(D)$, then $(u_{\ell-1},u_{\ell},u_1,u_{\ell-1})$ is a rainbow $3$-cycle. This implies that $(u_{\ell},u_1,u_{\ell-1})$ is a rainbow $(u_{\ell},u_{\ell-1})$-path and hence $(u_{\ell},u_{\ell-1})\in A(C_r(D))$. Note that $(u_{\ell-1},u_{\ell})\in A(C)$, which contradicts that $C$ has no symmetrical arc.

If $(u_{\ell-1},u_{1})\in A(D)$, since $(u_1,u_2)\in A(D)$, there exists $i\in \{3,4,\ldots,\ell-1\}$ such that $(u_i,u_1)\in A(D)$.
Let$$i_0=\mbox{min}\{i\in \{3,4,\ldots,\ell-1\}\,|\,(u_i,u_1)\in A(D)\}.$$
Considering $\{(u_{i_0-1},u_{i_0}),(u_{i_0},u_1)\}\subseteq A(D)$, we have $u_1,u_{i_0-1}$ are adjacent. By the choice of $i_0$, we have $(u_1,u_{i_0-1})\in A(D)$. It follows $(u_1,u_{i_0-1},u_{i_0},u_1)$ is a rainbow $3$-cycle. This implies that $(u_{i_0},u_1,u_{i_0-1})$ is a rainbow $(u_{i_0},u_{i_0-1})$-path and hence $(u_{i_0},u_{i_0-1})\in A(C_r(D))$. Note that $(u_{i_0-1},u_{i_0})\in A(C)$, which contradicts that that $C$ has no symmetrical arc.

Thus, $C_r(D)$ is a \emph{KP}-digraph.
\end{proof}

By Observation \ref{z1.4} and Theorem \ref{z4.3}, the following corollary is direct.

\begin{corollary}\label{z4.4}
Let $D$ be an $m$-arc-coloured quasi-transitive digraph with all $3$-cycles and all induced subdigraphs  $QT_4$ are rainbow in $D$. Then $D$ has a RP-kernel.
\end{corollary}

\section{Bipartite tournaments}
A digraph $D$ is a \emph{bipartite tournament} if there exists a partition of $V(D)$ into two sets $\{X,Y\}$ such that there exists no arc between any two vertices in the same set and there exists an arc between any two vertices in different sets. In this section, we consider the sufficient conditions for the existence of a RP-kernel in an arc-coloured bipartite tournament. We begin with two simple observations.

\begin{observation}\cite{11}\label{z3.6}
Let $D$ be an $m$-arc-coloured digraph and $v\in V(D)$ a source. Then $D$ has a RP-kernel if and only if $D-y$ has a RP-kernel.
\end{observation}
\begin{observation}\label{z3.1}
$1$-arc-coloured bipartite tournament $D=(X,Y)$ has a RP-kernel.
\end{observation}
\begin{proof} Obviously, a kernel of $D$ is also a RP-kernel of $D$. We claim that either $X$ or $Y$ is a kernel of $D$ and hence a RP-kernel of $D$. If $X$ is not a kernel, then there exists $v\in Y$ such that $X\Rightarrow v$. This implies that $Y$ is a kernel of $D$.
\end{proof}
In the following, we may assume $m\geqslant2$.

\begin{observation}
Let $D=(X,Y)$ be an $m$-arc-coloured bipartite tournament with $\min\{|X|,|Y|\}=1$. Then $D$ has a RP-kernel.
\end{observation}
\begin{proof}
W.l.o.g., assume $|X|=\min\{|X|,|Y|\}=1$. Obviously, $D$ has no cycle. This implies that $C_r(D)$ has also no cycle. By Theorem \ref{z1.1}, $C_r(D)$ has a unique kernel. By Observation \ref{z1.4}, $D$ has a RP-kernel.
\end{proof}
\begin{theorem}\label{z3.2}
Let $D=(X,Y)$ be an $m$-arc-coloured bipartite tournament with $\min\{|X|,|Y|\}=2$. If every $4$-cycle contained in $D$ is coloured with at least 3 colours, then $D$ has a RP-kernel.
\end{theorem}

\begin{proof}
W.l.o.g., assume $\min\{|X|,|Y|\}=|X|=2$ and $X=\{x_1,x_2\}$. If $X\Rightarrow Y$, then $Y$ is a RP-kernel of $D$. If $Y\Rightarrow X$, then $X$ is a RP-kernel of $D$. So we assume $X\nRightarrow Y$ and $Y\nRightarrow X$. Let
$$Y_0=\{y\in Y\,|\,X\Rightarrow y\},$$
$$Y_1=\{y\in Y\setminus Y_0\,|\,\mbox{there exists a rainbow }(y,Y_0)\mbox{-path in } D\},$$
$$Y_2= Y\setminus(Y_0\cup Y_1).$$

By Observation \ref{z3.6}, we also assume that $D$ has no source in $Y$. By the definition of $Y_0$, the following claim holds directly.
\vskip 2mm
\noindent {\bf Claim 1.} Each vertex in $Y\setminus Y_0=Y_1\cup Y_2$ has exactly one out-neighbour and one in-neighbour in $X$.
\vskip 2mm
We consider the following two cases.
\vskip 2mm
\noindent\emph{Case 1.} $Y_0\neq\emptyset$.

\vskip 2mm

Note that there exists no rainbow $(y,Y_0)$-path for any $y\in Y_2$. If $Y_2=\emptyset$, then $Y_0$ is a RP-kernel of $D$. So we assume that $Y_2\neq \emptyset$. Clearly, the following claim holds directly.

\vskip 2mm
\noindent {\bf Claim 2.} There exist rainbow paths from $X\cup Y_1$ to $Y_0$; there exists no rainbow path from $Y_0$ to $Y\setminus Y_0$ and there exists no rainbow path from $Y_2$ to $Y_0$.

\vskip 2mm
\noindent {\bf Claim 3.} For some $x_i\in X$, if $x_i$ has an in-neighbour in $Y_2$, then all arcs from $x_i$ to $Y_0$ are assigned the common colour.
\begin{proof} Suppose to the contrary that $C(x_i,y_1)\neq C(x_i,y_2)$ for some $y_1,y_2\in Y_0$. Let $y\in Y_2$ with $y\rightarrow x_i$. Since $C(x_i,y_1)\neq C(x_i,y_2)$, $y$ can reach $Y_0$ by a rainbow path passing through $(y,x_i)$ as well as either $(x_i,y_1)$ or $(x_i,y_2)$. This contradicts that there exists no rainbow $(y,Y_0)$-path for any $y\in Y_2$. Thus all arcs from $x_i$ to $Y_0$ are assigned the common colour.
\end{proof}
For convenience, we will denote the common colour assigned the arcs from $x_i$ to $Y_0$ by $C(x_i,Y_0)$ for $x_i\in X$ with an in-neighbour in $Y_2$. By the definition of $Y_2$, the following claim holds directly.

\vskip 2mm
\noindent {\bf Claim 4.} For any $y\in Y_2$ with $y\rightarrow x_i$ for some $x_i\in X$, $C(y,x_i)=C(x_i,Y_0)$.
\vskip 2mm

Let $S\subseteq Y_2$ be the maximal subset such that there exists no rainbow path for any pair of vertices of $S$ in $D$. Let
$$R=\{r\in Y_2\setminus S\,|\,\mbox{there exists no rainbow }(r,S)\mbox{-path in } D\}.$$

If $R=\emptyset$, then $Y_0\cup S$ is a RP-kernel of $D$. Assume that $R\neq \emptyset$ and let $r\in R\subseteq Y_2$ be arbitrary. By Claim $1$, w.l.o.g., we assume $$x_1\rightarrow r\rightarrow x_2.$$
By the choice of $S$, there exists a rainbow $(s,r)$-path $P$ for some $s\in S$ in $D$.

If $\ell(P)=4$, w.l.o.g., assume $P=(s,x_2,y,x_1,r)$ where $y\in Y\setminus Y_0$. It is clear that $C(x_1,r)\neq C(x_1,Y_0)$, since otherwise, in the rainbow path $P$, we replace the arc $(x_1,r)$ with $(x_1,y_0)$ for any $y_0\in Y_0$ and get a rainbow $(s,y_0)$-path, which contradicts $s\in Y_2$. Now we claim that $Y_0\cup\{r\}$ is a RP-kernel of $D$. By Claim 2, it is sufficient to show that there exists a rainbow $(z,r)$-path for any $z\in Y_2\setminus \{s,r\}$. By Claim $1$,  we have either $z\rightarrow x_1$ or $z\rightarrow x_2$. If $z\rightarrow x_1$, by Claim $4$, we have $C(z,x_1)=C(x_1,Y_0)$. Combining with $C(x_1,r)\neq C(x_1,Y_0)$, we have $(z,x_1,r)$ is a rainbow $(z,r)$-path. If $z\rightarrow x_2$, by Claim $4$, we have $C(z,x_2)=C(s,x_2)=C(x_2,Y_0)$. In the rainbow path $P$, we replace the arc $(s,x_2)$ with $(z,x_2)$ and get a rainbow $(z,r)$-path $(z,x_2,y,x_1,r)$. This implies that $Y_0\cup\{r\}$ is a RP-kernel of $D$.

If $\ell(P)=2$, now $(s,x_1,r)$ is the rainbow $(s,r)$-path. Note that $s\in S\subseteq Y_2$ and $s\rightarrow x_1$. Let $y\in Y_2$ with $y\rightarrow x_1$. By Claim $4$,  we have $C(y,x_1)=C(s,x_1)=C(x_1,Y_0)$. In the rainbow path $(s,x_1,r)$, we replace the arc $(s,x_1)$ with $(y,x_1)$ and get a rainbow path $(y,x_1,r)$. This means that all vertices dominating $x_1$ in $Y_2$ can reach $r$ by a rainbow path. Let $$Q_1=\{y\in Y_2\setminus\{r\}\,|\,y\rightarrow x_1\},\,\,Q_2=\{y\in Y_2\setminus\{r\}\,|\,y\rightarrow x_2\}.$$
Clearly, each vertex of $Q_1$ can reach $r$ by a rainbow path. By Claim $1$, we have $Q_1\cup Q_2=Y_2\setminus\{r\}$ and $Q_1\cap Q_2=\emptyset$. If $Q_2=\emptyset$, then $Y_0\cup\{r\}$ is a RP-kernel of $D$. So assume that $Q_2\neq\emptyset$. Also by Claim 1,
$$x_1\Rightarrow Q_2\Rightarrow x_2.$$

If there exists a rainbow $(q_2,r)$-path $P'$ for some $q_2\in Q_2$, we claim that $Y_0\cup\{r\}$ is a RP-kernel of $D$. Since $x_1\rightarrow r\rightarrow x_2$ and $x_1\rightarrow q_2\rightarrow x_2$, we have $\ell(P')\neq 2$ and hence $\ell(P')=4$. W.l.o.g., assume $P'=(q_2,x_2,y,x_1,r)$ where $y\in Y\setminus Y_0$. It is sufficient to show that there exists a rainbow $(q'_2,r)$-path for any $q'_2\in Q_2\setminus \{q_2,r\}$. By Claim $4$, we have $C(q_2,x_2)=C(q'_2,x_2)=C(x_2,Y_0)$. In the rainbow path $P'$, we replace the arc $(q_2,x_2)$ with $(q'_2,x_2)$ and get a rainbow $(q'_2,r)$-path $(q'_2,x_2,y,x_1,r)$. This implies that $Y_0\cup\{r\}$ is a RP-kernel of $D$.

If there exists a rainbow $(r,q_2)$-path $P''$ for some $q_2\in Q_2$, we claim that $Y_0\cup\{q_2\}$ is a RP-kernel of $D$. Since $x_1\rightarrow r\rightarrow x_2$ and $x_1\rightarrow q_2\rightarrow x_2$, we have $\ell(P'')\neq 2$ and hence $\ell(P'')=4$. W.l.o.g., assume $P''=(r,x_2,y,x_1,q_2)$ where $y\in Y\setminus Y_0$. It is clear that $C(x_1,q_2)\neq C(x_1,Y_0)$, since otherwise, in the rainbow path $P''$, we replace the arc $(x_1,q_2)$ with $(x_1,y_0)$ for any $y_0\in Y_0$ and get a rainbow $(r,y_0)$-path, which contradicts $r\in Y_2$. Now it is sufficient to show that there exists a rainbow $(w,q_2)$-path for any $w\in Y_2\setminus \{r,q_2\}$. By Claim $1$, we have either $w\rightarrow x_1$ or $w\rightarrow x_2$. If $w\rightarrow x_1$, by Claim $4$, we have $C(w,x_1)=C(x_1,Y_0)$. Combining with $C(x_1,q_2)\neq C(x_1,Y_0)$, we have $(w,x_1,q_2)$ is a rainbow $(w,q_2)$-path. If $w\rightarrow x_2$, by Claim $4$, we have $C(w,x_2)=C(r,x_2)=C(x_2,Y_0)$. In the rainbow path $P''$, we replace the arc $(r,x_2)$ with $(w,x_2)$ and get a rainbow $(w,q_2)$-path $(w,x_2,y,x_1,q_2)$. This implies that $Y_0\cup\{q_2\}$ is a RP-kernel of $D$.

If there exists no rainbow $(q_2,r)$-path and no rainbow $(r,q_2)$-path for any $q_2\in Q_2$, we claim that $Y_0\cup Q_2\cup\{r\}$ is a RP-kernel of $D$.
\vskip 2mm
\noindent {\bf Claim 5.} If there exists no rainbow $(Q_2,r)$-path and there exists no rainbow $(r,Q_2)$-path, then there exists no rainbow path for any pair of vertices of $Q_2$.

\begin{proof}Suppose to the contrary that there exists a rainbow path for some $q_2,\,q'_2\in Q_2$, say $(q_2,x_2,y,x_1,q'_2)$ where $y\in Y\setminus Y_0$. Note that $y\neq r$ since there exists no rainbow $(Q_2,r)$-path. Since $r,q_2\in Y_2$, by Claim $4$, we have $C(r,x_2)=C(q_2,x_2)=C(x_2,Y_0)$. In the rainbow path $(q_2,x_2,y,x_1,q'_2)$, we replace the arc $(q_2,x_2)$ with $(r,x_2)$ and get a rainbow $(r,q'_2)$-path $(r,x_2,y,x_1,q'_2)$, which contradicts that there exists no rainbow $(r,Q_2)$-path.
\end{proof}

Recall that $Q_1\cup Q_2=Y_2\setminus\{r\}$ and each vertex of $Q_1$ can reach $r$ by a rainbow path. By Claim 2 and $5$, $Y_0\cup Q_2\cup\{r\}$ is a RP-kernel of $D$.
\vskip 2mm
\noindent \emph{Case 2.} $Y_0=\emptyset$.
\vskip 2mm

By Claim 1, in this case each vertex of $Y$ has one out-neighbour and one in-neighbour in $X$. We give a partition of $Y$ as follows.
$$Y'=\{y\in Y\,|\,x_1\rightarrow y\rightarrow x_2\},\quad Y''=\{y\in Y\,|\,x_2\rightarrow y\rightarrow x_1\}.$$

If $Y'=\emptyset$, then $Y=Y''$ and $x_2\Rightarrow Y\Rightarrow x_1$. If there exists a rainbow $(x_2,x_1)$-path, then $\{x_1\}$ is a RP-kernel of $D$. If there exists no rainbow $(x_2,x_1)$-path, then $\{x_1,x_2\}$ is a RP-kernel of $D$. If $Y''=\emptyset$, we can prove that either $\{x_2\}$ or $\{x_1,x_2\}$ is a RP-kernel of $D$. So we assume $Y'\neq\emptyset$ and $Y''\neq\emptyset$.

In particular, we consider the following subsets of $Y'$ and $Y''$, respectively.
$$Y^*=\{y\in Y'\,|\,C(x_1,y)\neq C(y,x_2)\}, \quad Y^{**}=\{y\in Y''\,|\,C(x_2,y)\neq C(y,x_1)\}.$$

Let $y'\in Y'$ and $y''\in Y''$ be arbitrary. If $Y^*=\emptyset$, then $C(x_1,y')=C(y',x_2)$. Note that $(x_1,y',x_2,y'',x_1)$ is a 4-cycle. Since every $4$-cycle is coloured with at least 3 colours, we have $(y',x_2,y'',x_1)$ is a rainbow path. Clearly, $(y'',x_1)$ is a rainbow path. It follows that $\{x_1\}$ is a RP-kernel of $D$. Similarly, If $Y^{**}=\emptyset$, we can prove that $\{x_2\}$ is a RP-kernel of $D$. So we assume $Y^*\neq\emptyset$ and $Y^{**}\neq\emptyset$.

\vskip 2mm

\noindent\emph{Subcase 2.1.} $Y'\setminus Y^* \neq\emptyset$ or $Y''\setminus Y^{**} \neq\emptyset$.

\vskip 2mm

W.l.o.g., assume $Y'\setminus Y^* \neq\emptyset$. Let $y'\in Y'\setminus Y^*$ be arbitrary. We assume that $C(x_1,y')=C(y',x_2)=\alpha$. Let
$$Y'_{\alpha}=\{y'\in Y'\,|\,C(x_1,y')=C(y',x_2)=\alpha\}.$$
Clearly, $Y'_{\alpha}\neq\emptyset$. Let $y''\in Y''$ be arbitrary. Note that $x_1\Rightarrow y'\Rightarrow x_2\Rightarrow y''\Rightarrow x_1$. Since every $4$-cycle in $D$ is coloured with at least $3$ colours, we have $C(x_2,y'')\ne C(y'',x_1)$, $C(x_2,y'')\ne\alpha$ and $C(y'',x_1)\ne\alpha$. Let $C(x_2,y'')=\beta$ and $C(y'',x_1)=\gamma$. Then $\alpha,\,\beta,\,\gamma$ are pairwise distinct. Let $$Y''_{\beta\gamma}=\{y''\in Y''\,|\,C(x_2,y'')=\beta,\,C(y'',x_1)=\gamma \}.$$

Define the following vertex subsets, see Figure 2, in which a box represent a set of vertices and dotted arcs, dashed arcs, thick dotted arcs, solid arcs represent respectively the arcs coloured by $\alpha$, $\beta$, $\gamma$ and a colour not in $\{\alpha,\beta,\gamma\}$.
 $$Y'_{\alpha\beta}=\{y'\in Y'\,|\,C(x_1,y')=\alpha,\,C(y',x_2)=\beta\},$$
 $$Y'_{\beta\alpha}=\{y'\in Y'\,|\,C(x_1,y')=\beta,\,C(y',x_2)=\alpha\},$$
 $$Y'_{\alpha\gamma}=\{y'\in Y'\,|\,C(x_1,y')=\alpha,\,C(y',x_2)=\gamma\},$$
 $$Y'_{\gamma\alpha}=\{y'\in Y'\,|\,C(x_1,y')=\gamma,\,C(y',x_2)=\alpha\},$$
 $$Y'_{\omega\alpha}=\{y'\in Y'\,|\,C(x_1,y')\notin\{\alpha,\beta,\gamma\},\,C(y',x_2)=\alpha\},$$
 $$Y'_{\omega\beta}=\{y'\in Y'\,|\,C(x_1,y')\notin\{\alpha,\beta,\gamma\},\,C(y',x_2)=\beta\},$$
$$Y'_{\omega\gamma}=\{y'\in Y'\,|\,C(x_1,y')\notin\{\alpha,\beta,\gamma\},\,C(y',x_2)=\gamma\},$$
$$Y'^{+}_{\omega}=\{y'\in Y'\,|\,C(x_1,y')\mbox{ is arbitrary, }\,C(y',x_2)\notin\{\alpha,\beta,\gamma\}\}.$$
$$Y''_{\gamma\beta}=\{y''\in Y''\,|\,C(x_2,y'')=\gamma,\,C(y'',x_1)=\beta \},$$
$$Y''_{\gamma\omega}=\{y''\in Y''\,|\,C(x_2,y'')=\gamma,\,C(y'',x_1)\notin\{\alpha,\beta,\gamma\}\},$$
$$Y''_{\omega\gamma}=\{y''\in Y''\,|\,C(x_2,y'')\notin\{\alpha,\beta,\gamma\},\,C(y'',x_1)=\gamma \},$$
$$Y''_{\beta\omega}=\{y''\in Y''\,|\,C(x_2,y'')=\beta,\,C(y'',x_1)\notin\{\alpha,\beta,\gamma\}\},$$
$$Y''_{\omega\beta}=\{y''\in Y''\,|\,C(x_2,y'')\notin\{\alpha,\beta,\gamma\},\,C(y'',x_1)=\beta \},$$
$$Y''_{\omega_1\omega_2}=\{y''\in Y''\,|\,C(x_2,y''), C(y'',x_1)\notin\{\alpha,\beta,\gamma\}\mbox{ and }C(x_2,y'')\neq C(y'',x_1)\}.$$

\begin {figure}[h]
\unitlength0.4cm
\begin{center}
\begin{picture}(30,12)
\put(17,11){\circle*{.3}}
\put(16.9,11.5){$x_1$}
\put(17,0){\circle*{.3}}
\put(16.9,-0.7){$x_2$}
\put(1,5){\framebox(1.4,1)}
\put(1,5.25){\small{$Y'^{+}_{\omega}$}}
\put(2.7,5){\framebox(1.4,1)}
\put(2.7,5.25){\small{$Y'_{\omega\gamma}$}}
\put(4.4,5){\framebox(1.4,1)}
\put(4.4,5.25){\small{$Y'_{\omega\beta}$}}
\put(6.1,5){\framebox(1.4,1)}
\put(6.1,5.25){\small{$Y'_{\omega\alpha}$}}
\put(7.8,5){\framebox(1.4,1)}
\put(7.8,5.25){\small{$Y'_{\gamma\alpha}$}}
\put(9.5,5){\framebox(1.4,1)}
\put(9.5,5.25){\small{$Y'_{\alpha\gamma}$}}
\put(11.2,5){\framebox(1.4,1)}
\put(11.2,5.25){\small{$Y'_{\beta\alpha}$}}
\put(12.9,5){\framebox(1.4,1)}
\put(12.9,5.25){\small{$Y'_{\alpha\beta}$}}
\put(14.6,5){\framebox(1.4,1)}
\put(14.6,5.25){\small{$Y'_{\alpha}$}}

\qbezier(17,0)(9.25,2.5)(1.5,5)
\put(6.5,3.4){\vector(3,-1){.01}}
\qbezier(17,11)(10.1,8.5)(3.2,6)
\put(7.8,7.7){\vector(-3,-1){.01}}
\put(3.45,4.9){\circle*{.03}}
\put(3.73,4.8){\circle*{.03}}
\put(4.01,4.7){\circle*{.03}}
\put(4.29,4.6){\circle*{.03}}
\put(4.57,4.5){\circle*{.03}}
\put(4.85,4.4){\circle*{.03}}
\put(5.13,4.3){\circle*{.03}}
\put(5.41,4.2){\circle*{.03}}
\put(5.69,4.1){\circle*{.03}}
\put(5.97,4){\circle*{.03}}
\put(6.25,3.9){\circle*{.03}}
\put(6.53,3.8){\circle*{.03}}
\put(6.81,3.7){\circle*{.03}}
\put(7.09,3.6){\circle*{.03}}
\put(7.37,3.5){\circle*{.03}}
\put(7.65,3.4){\circle*{.03}}
\put(7.93,3.3){\circle*{.03}}
\put(8.21,3.2){\circle*{.03}}
\put(8.49,3.1){\circle*{.03}}
\put(8.77,3){\circle*{.03}}
\put(9.05,2.9){\circle*{.03}}
\put(9.33,2.8){\circle*{.03}}
\put(9.61,2.7){\circle*{.03}}
\put(9.87,2.6){\circle*{.03}}
\put(10.15,2.5){\circle*{.03}}
\put(10.42,2.4){\circle*{.03}}
\put(10.7,2.3){\circle*{.03}}
\put(10.95,2.2){\circle*{.03}}
\put(11.2,2.1){\circle*{.03}}
\put(11.45,2){\circle*{.03}}
\put(11.7,1.9){\circle*{.03}}
\put(11.98,1.8){\circle*{.03}}
\put(12.26,1.7){\circle*{.03}}
\put(12.54,1.6){\circle*{.03}}
\put(12.82,1.5){\circle*{.03}}
\put(13.1,1.4){\circle*{.03}}
\put(13.38,1.3){\circle*{.03}}
\put(13.66,1.2){\circle*{.03}}
\put(13.94,1.1){\circle*{.03}}
\put(14.22,1){\circle*{.03}}
\put(14.5,0.9){\circle*{.03}}
\put(14.78,0.8){\circle*{.03}}
\put(15.06,0.7){\circle*{.03}}
\put(15.34,0.6){\circle*{.03}}
\put(15.62,0.5){\circle*{.03}}
\put(15.9,0.4){\circle*{.03}}
\put(16.18,0.3){\circle*{.03}}
\put(16.46,0.2){\circle*{.03}}

\put(7.81,3.3){\vector(2,-1){.01}}
\qbezier(17,11)(10.95,8.5)(4.9,6)
\put(9.1,7.7){\vector(-2,-1){.01}}
\multiput(4.9,5)(.02,-.01){20}{\line(1,0){.037182}}
\multiput(5.7,4.7)(.02,-.01){20}{\line(1,0){.037182}}
\multiput(6.4,4.4)(.02,-.01){20}{\line(1,0){.037182}}
\multiput(7.1,4.1)(.02,-.01){20}{\line(1,0){.037182}}
\multiput(7.8,3.8)(.02,-.01){20}{\line(1,0){.037182}}
\multiput(8.5,3.5)(.02,-.01){20}{\line(1,0){.037182}}
\multiput(9.3,3.2)(.02,-.01){20}{\line(1,0){.037182}}
\multiput(10,2.9)(.02,-.01){20}{\line(1,0){.037182}}
\multiput(10.7,2.6)(.02,-.01){20}{\line(1,0){.037182}}
\multiput(11.45,2.3)(.02,-.01){20}{\line(1,0){.037182}}
\multiput(12.2,2)(.02,-.01){20}{\line(1,0){.037182}}
\multiput(12.9,1.7)(.02,-.01){20}{\line(1,0){.037182}}
\multiput(13.6,1.4)(.02,-.01){20}{\line(1,0){.037182}}
\multiput(14.35,1.1)(.02,-.01){20}{\line(1,0){.037182}}
\multiput(15.1,0.8)(.02,-.01){20}{\line(1,0){.037182}}
\multiput(15.8,0.5)(.02,-.01){20}{\line(1,0){.037182}}
\put(9.1,3.25){\vector(2,-1){.01}}
\qbezier(17,11)(11.8,8.5)(6.6,6)
\put(10.4,7.8){\vector(-2,-1){.01}}
\qbezier[45](17,0)(11.8,2.5)(6.6,5)
\put(10.4,3.15){\vector(2,-1){.01}}

\put(16.71,10.85){\circle*{.03}}
\put(16.45,10.7){\circle*{.03}}
\put(16.19,10.55){\circle*{.03}}
\put(15.93,10.4){\circle*{.03}}
\put(15.67,10.25){\circle*{.03}}
\put(15.41,10.1){\circle*{.03}}
\put(15.15,9.95){\circle*{.03}}
\put(14.89,9.8){\circle*{.03}}
\put(14.63,9.65){\circle*{.03}}
\put(14.36,9.5){\circle*{.03}}
\put(14.1,9.35){\circle*{.03}}
\put(13.84,9.2){\circle*{.03}}
\put(13.58,9.05){\circle*{.03}}
\put(13.32,8.9){\circle*{.03}}
\put(13.07,8.75){\circle*{.03}}
\put(12.82,8.6){\circle*{.03}}
\put(12.57,8.45){\circle*{.03}}
\put(12.32,8.3){\circle*{.03}}
\put(12.06,8.15){\circle*{.03}}
\put(11.8,8){\circle*{.03}}
\put(11.53,7.85){\circle*{.03}}
\put(11.25,7.7){\circle*{.03}}
\put(10.97,7.55){\circle*{.03}}
\put(10.73,7.4){\circle*{.03}}
\put(10.45,7.25){\circle*{.03}}
\put(10.2,7.1){\circle*{.03}}
\put(9.95,6.95){\circle*{.03}}
\put(9.7,6.8){\circle*{.03}}
\put(9.45,6.65){\circle*{.03}}
\put(9.18,6.5){\circle*{.03}}
\put(8.9,6.35){\circle*{.03}}
\put(8.65,6.2){\circle*{.03}}
\put(8.4,6.05){\circle*{.03}}
\put(11.3,7.7){\vector(-2,-1){.01}}
\qbezier[40](17,0)(12.65,2.5)(8.3,5)
\put(11.3,3.3){\vector(2,-1){.01}}
\qbezier[40](17,11)(13.5,8.5)(10,6)
\put(12.4,7.7){\vector(-3,-2){.01}}
\put(10.15,4.9){\circle*{.03}}
\put(10.37,4.75){\circle*{.03}}
\put(10.58,4.6){\circle*{.03}}
\put(10.79,4.45){\circle*{.03}}
\put(11,4.3){\circle*{.03}}
\put(11.21,4.15){\circle*{.03}}
\put(11.42,4){\circle*{.03}}
\put(11.63,3.85){\circle*{.03}}
\put(11.84,3.7){\circle*{.03}}
\put(12.05,3.55){\circle*{.03}}
\put(12.26,3.4){\circle*{.03}}
\put(12.47,3.25){\circle*{.03}}
\put(12.68,3.1){\circle*{.03}}
\put(12.89,2.95){\circle*{.03}}
\put(13.1,2.8){\circle*{.03}}
\put(13.31,2.65){\circle*{.03}}
\put(13.52,2.5){\circle*{.03}}
\put(13.73,2.35){\circle*{.03}}
\put(13.94,2.2){\circle*{.03}}
\put(14.15,2.05){\circle*{.03}}
\put(14.36,1.9){\circle*{.03}}
\put(14.57,1.75){\circle*{.03}}
\put(14.79,1.6){\circle*{.03}}
\put(15,1.45){\circle*{.03}}
\put(15.21,1.3){\circle*{.03}}
\put(15.42,1.15){\circle*{.03}}
\put(15.63,1){\circle*{.03}}
\put(15.84,0.85){\circle*{.03}}
\put(16.05,0.7){\circle*{.03}}
\put(16.27,0.55){\circle*{.03}}
\put(16.48,0.4){\circle*{.03}}
\put(16.69,0.25){\circle*{.03}}
\put(12.4,3.3){\vector(3,-2){.01}}

\multiput(17,11)(-.02,-.02){15}{\line(1,0){.037182}}
\multiput(16.55,10.6)(-.02,-.02){15}{\line(1,0){.037182}}
\multiput(16.1,10.2)(-.02,-.02){15}{\line(1,0){.037182}}
\multiput(15.66,9.8)(-.02,-.02){15}{\line(1,0){.037182}}
\multiput(15.25,9.4)(-.02,-.02){15}{\line(1,0){.037182}}
\multiput(14.8,9)(-.02,-.02){15}{\line(1,0){.037182}}
\multiput(14.4,8.6)(-.02,-.02){15}{\line(1,0){.037182}}
\multiput(13.95,8.2)(-.02,-.02){15}{\line(1,0){.037182}}
\multiput(13.55,7.8)(-.02,-.02){15}{\line(1,0){.037182}}
\multiput(13.15,7.4)(-.02,-.02){15}{\line(1,0){.037182}}
\multiput(12.75,7)(-.02,-.02){15}{\line(1,0){.037182}}
\multiput(12.35,6.6)(-.02,-.02){15}{\line(1,0){.037182}}
\multiput(11.9,6.2)(-.02,-.02){10}{\line(1,0){.037182}}
\put(13.5,7.7){\vector(-1,-1){.01}}
\qbezier[30](17,0)(14.35,2.5)(11.7,5)
\put(13.5,3.3){\vector(1,-1){.01}}
\qbezier[25](17,11)(15.2,8.5)(13.4,6)
\put(14.6,7.7){\vector(-2,-3){.01}}
\multiput(13.4,5)(.02,-.03){12}{\line(1,0){.037182}}
\multiput(13.75,4.5)(.02,-.03){12}{\line(1,0){.037182}}
\multiput(14.1,4)(.02,-.03){12}{\line(1,0){.037182}}
\multiput(14.45,3.5)(.02,-.03){12}{\line(1,0){.037182}}
\multiput(14.85,3)(.02,-.03){12}{\line(1,0){.037182}}
\multiput(15.2,2.5)(.02,-.03){12}{\line(1,0){.037182}}
\multiput(15.55,2)(.02,-.03){12}{\line(1,0){.037182}}
\multiput(15.9,1.5)(.02,-.03){12}{\line(1,0){.037182}}
\multiput(16.25,1)(.02,-.03){12}{\line(1,0){.037182}}
\multiput(16.6,0.5)(.02,-.03){12}{\line(1,0){.037182}}
\put(14.6,3.3){\vector(2,-3){.01}}
\qbezier[25](17,11)(16.05,8.5)(15.1,6)
\put(15.75,7.7){\vector(-1,-3){.01}}
\qbezier[25](17,0)(16.05,2.5)(15.1,5)
\put(15.75,3.3){\vector(1,-3){.01}}

\put(18.5,5){\framebox(1.4,1)}
\put(18.5,5.25){\small{$Y''_{\beta\gamma}$}}
\put(20.2,5){\framebox(1.4,1)}
\put(20.2,5.25){\small{$Y''_{\gamma\beta}$}}
\put(21.9,5){\framebox(1.4,1)}
\put(21.9,5.25){\small{$Y''_{\gamma\omega}$}}
\put(23.6,5){\framebox(1.4,1)}
\put(23.6,5.25){\small{$Y''_{\omega\gamma}$}}
\put(25.3,5){\framebox(1.4,1)}
\put(25.3,5.25){\small{$Y''_{\beta\omega}$}}
\put(27,5){\framebox(1.4,1)}
\put(27,5.25){\small{$Y''_{\omega\beta}$}}
\put(28.7,5){\framebox(1.4,1)}
\put(28.7,5.25){\small{$Y''_{\omega_1\omega_2}$}}
\put(18.91,6.2){\circle*{.03}}
\put(18.83,6.4){\circle*{.03}}
\put(18.75,6.6){\circle*{.03}}
\put(18.67,6.8){\circle*{.03}}
\put(18.59,7){\circle*{.03}}
\put(18.51,7.2){\circle*{.03}}
\put(18.43,7.4){\circle*{.03}}
\put(18.35,7.6){\circle*{.03}}
\put(18.27,7.8){\circle*{.03}}
\put(18.19,8){\circle*{.03}}
\put(18.11,8.2){\circle*{.03}}
\put(18.03,8.4){\circle*{.03}}
\put(17.95,8.6){\circle*{.03}}
\put(17.87,8.8){\circle*{.03}}
\put(17.79,9){\circle*{.03}}
\put(17.71,9.2){\circle*{.03}}
\put(17.63,9.4){\circle*{.03}}
\put(17.55,9.6){\circle*{.03}}
\put(17.47,9.8){\circle*{.03}}
\put(17.39,10){\circle*{.03}}
\put(17.31,10.2){\circle*{.03}}
\put(17.23,10.4){\circle*{.03}}
\put(17.15,10.6){\circle*{.03}}
\put(17.07,10.8){\circle*{.03}}

\put(18.2,8){\vector(-1,3){.01}}
\multiput(17,0)(.01,.03){12}{\line(1,0){.037182}}
\multiput(17.2,0.5)(.01,.03){12}{\line(1,0){.037182}}
\multiput(17.4,1)(.01,.03){12}{\line(1,0){.037182}}
\multiput(17.6,1.5)(.01,.03){12}{\line(1,0){.037182}}
\multiput(17.8,2)(.01,.03){12}{\line(1,0){.037182}}
\multiput(18,2.5)(.01,.03){12}{\line(1,0){.037182}}
\multiput(18.2,3)(.01,.03){12}{\line(1,0){.037182}}
\multiput(18.4,3.5)(.01,.03){12}{\line(1,0){.037182}}
\multiput(18.6,4)(.01,.03){12}{\line(1,0){.037182}}
\multiput(18.8,4.5)(.01,.03){12}{\line(1,0){.037182}}
\put(18.3,3.3){\vector(1,3){.01}}

\multiput(17,11)(.02,-.03){12}{\line(1,0){.037182}}
\multiput(17.35,10.5)(.02,-.03){12}{\line(1,0){.037182}}
\multiput(17.71,10)(.02,-.03){12}{\line(1,0){.037182}}
\multiput(18.08,9.5)(.02,-.03){12}{\line(1,0){.037182}}
\multiput(18.44,9)(.02,-.03){12}{\line(1,0){.037182}}
\multiput(18.81,8.5)(.02,-.03){12}{\line(1,0){.037182}}
\multiput(19.18,8)(.02,-.03){12}{\line(1,0){.037182}}
\multiput(19.55,7.5)(.02,-.03){12}{\line(1,0){.037182}}
\multiput(19.92,7)(.02,-.03){12}{\line(1,0){.037182}}
\multiput(20.29,6.5)(.02,-.03){12}{\line(1,0){.037182}}
\put(19.2,8){\vector(-2,3){.01}}

\put(17.15,0.2){\circle*{.03}}
\put(17.3,0.4){\circle*{.03}}
\put(17.45,0.6){\circle*{.03}}
\put(17.6,0.8){\circle*{.03}}
\put(17.75,1){\circle*{.03}}
\put(17.9,1.2){\circle*{.03}}
\put(18.05,1.4){\circle*{.03}}
\put(18.2,1.6){\circle*{.03}}
\put(18.35,1.8){\circle*{.03}}
\put(18.5,2){\circle*{.03}}
\put(18.65,2.2){\circle*{.03}}
\put(18.8,2.4){\circle*{.03}}
\put(18.95,2.6){\circle*{.03}}
\put(19.1,2.8){\circle*{.03}}
\put(19.25,3){\circle*{.03}}
\put(19.4,3.2){\circle*{.03}}
\put(19.55,3.4){\circle*{.03}}
\put(19.7,3.6){\circle*{.03}}
\put(19.85,3.8){\circle*{.03}}
\put(20,4){\circle*{.03}}
\put(20.15,4.2){\circle*{.03}}
\put(20.3,4.4){\circle*{.03}}
\put(20.45,4.6){\circle*{.03}}
\put(20.6,4.8){\circle*{.03}}
\put(19.45,3.3){\vector(2,3){.01}}

\qbezier(17,11)(19.75,8.5)(22.5,6)
\put(20.3,8){\vector(-1,1){.01}}
\put(17.2,0.2){\circle*{.03}}
\put(17.4,0.38){\circle*{.03}}
\put(17.6,0.56){\circle*{.03}}
\put(17.8,0.74){\circle*{.03}}
\put(18,0.92){\circle*{.03}}
\put(18.2,1.1){\circle*{.03}}
\put(18.4,1.28){\circle*{.03}}
\put(18.6,1.46){\circle*{.03}}
\put(18.8,1.64){\circle*{.03}}
\put(19,1.82){\circle*{.03}}
\put(19.2,2){\circle*{.03}}
\put(19.4,2.18){\circle*{.03}}
\put(19.6,2.36){\circle*{.03}}
\put(19.8,2.54){\circle*{.03}}
\put(20,2.72){\circle*{.03}}
\put(20.2,2.9){\circle*{.03}}
\put(20.4,3.08){\circle*{.03}}
\put(20.6,3.26){\circle*{.03}}
\put(20.8,3.44){\circle*{.03}}
\put(21,3.62){\circle*{.03}}
\put(21.2,3.8){\circle*{.03}}
\put(21.4,3.98){\circle*{.03}}
\put(21.6,4.16){\circle*{.03}}
\put(21.8,4.34){\circle*{.03}}
\put(22,4.52){\circle*{.03}}
\put(22.2,4.7){\circle*{.03}}
\put(22.4,4.88){\circle*{.03}}

\put(20.65,3.3){\vector(1,1){.01}}
\put(24.25,6.05){\circle*{.03}}
\put(24,6.2){\circle*{.03}}
\put(23.81,6.35){\circle*{.03}}
\put(23.59,6.5){\circle*{.03}}
\put(23.37,6.65){\circle*{.03}}
\put(23.15,6.8){\circle*{.03}}
\put(22.93,6.95){\circle*{.03}}
\put(22.71,7.1){\circle*{.03}}
\put(22.49,7.25){\circle*{.03}}
\put(22.27,7.4){\circle*{.03}}
\put(22.05,7.55){\circle*{.03}}
\put(21.83,7.7){\circle*{.03}}
\put(21.61,7.85){\circle*{.03}}
\put(21.39,8){\circle*{.03}}
\put(21.17,8.15){\circle*{.03}}
\put(20.95,8.3){\circle*{.03}}
\put(20.73,8.45){\circle*{.03}}
\put(20.51,8.6){\circle*{.03}}
\put(20.29,8.75){\circle*{.03}}
\put(20.07,8.9){\circle*{.03}}
\put(19.85,9.05){\circle*{.03}}
\put(19.63,9.2){\circle*{.03}}
\put(19.41,9.35){\circle*{.03}}
\put(19.19,9.5){\circle*{.03}}
\put(18.97,9.65){\circle*{.03}}
\put(18.75,9.8){\circle*{.03}}
\put(18.53,9.95){\circle*{.03}}
\put(18.31,10.1){\circle*{.03}}
\put(18.09,10.25){\circle*{.03}}
\put(17.87,10.4){\circle*{.03}}
\put(17.65,10.55){\circle*{.03}}
\put(17.43,10.7){\circle*{.03}}

\put(21.37,8){\vector(-3,2){.01}}
\qbezier(17,0)(20.65,2.5)(24.3,5)
\put(21.85,3.3){\vector(3,2){.01}}
\qbezier(17,11)(21.5,8.5)(26,6)
\put(22.4,8){\vector(-3,2){.01}}
\multiput(17,0)(.02,.01){15}{\line(1,0){.037182}}
\multiput(17.44,0.25)(.02,.01){15}{\line(1,0){.037182}}
\multiput(17.88,0.5)(.02,.01){15}{\line(1,0){.037182}}
\multiput(18.32,0.75)(.02,.01){15}{\line(1,0){.037182}}
\multiput(18.76,1)(.02,.01){15}{\line(1,0){.037182}}
\multiput(19.2,1.25)(.02,.01){15}{\line(1,0){.037182}}
\multiput(19.65,1.5)(.02,.01){15}{\line(1,0){.037182}}
\multiput(20.1,1.75)(.02,.01){15}{\line(1,0){.037182}}
\multiput(20.58,2)(.02,.01){15}{\line(1,0){.037182}}
\multiput(21,2.25)(.02,.01){15}{\line(1,0){.037182}}
\multiput(21.48,2.5)(.02,.01){15}{\line(1,0){.037182}}
\multiput(21.94,2.75)(.02,.01){15}{\line(1,0){.037182}}
\multiput(22.35,3)(.02,.01){15}{\line(1,0){.037182}}
\multiput(22.8,3.25)(.02,.01){15}{\line(1,0){.037182}}
\multiput(23.25,3.5)(.02,.01){15}{\line(1,0){.037182}}
\multiput(23.7,3.75)(.02,.01){15}{\line(1,0){.037182}}
\multiput(24.15,4)(.02,.01){15}{\line(1,0){.037182}}
\multiput(24.6,4.25)(.02,.01){15}{\line(1,0){.037182}}
\multiput(25.05,4.5)(.02,.01){15}{\line(1,0){.037182}}
\multiput(25.5,4.75)(.02,.01){15}{\line(1,0){.037182}}

\put(22.95,3.3){\vector(2,1){.01}}

\multiput(27.7,6)(-.02,.01){20}{\line(1,0){.037182}}
\multiput(27,6.3)(-.02,.01){20}{\line(1,0){.037182}}
\multiput(26.3,6.6)(-.02,.01){20}{\line(1,0){.037182}}
\multiput(25.65,6.9)(-.02,.01){20}{\line(1,0){.037182}}
\multiput(25,7.2)(-.02,.01){20}{\line(1,0){.037182}}
\multiput(24.35,7.5)(-.02,.01){20}{\line(1,0){.037182}}
\multiput(23.65,7.8)(-.02,.01){20}{\line(1,0){.037182}}
\multiput(23.05,8.1)(-.02,.01){20}{\line(1,0){.037182}}
\multiput(22.4,8.4)(-.02,.01){20}{\line(1,0){.037182}}
\multiput(21.8,8.7)(-.02,.01){20}{\line(1,0){.037182}}
\multiput(21.15,9)(-.02,.01){20}{\line(1,0){.037182}}
\multiput(20.55,9.3)(-.02,.01){20}{\line(1,0){.037182}}
\multiput(19.9,9.6)(-.02,.01){20}{\line(1,0){.037182}}
\multiput(19.25,9.9)(-.02,.01){20}{\line(1,0){.037182}}
\multiput(18.6,10.2)(-.02,.01){20}{\line(1,0){.037182}}
\multiput(17.95,10.5)(-.02,.01){20}{\line(1,0){.037182}}
\multiput(17.3,10.8)(-.02,.01){20}{\line(1,0){.037182}}
\put(23.3,8){\vector(-2,1){.01}}
\qbezier(17,0)(22.1,2.5)(27.7,5)
\put(24,3.3){\vector(2,1){.01}}
\qbezier(17,11)(23.2,8.5)(29.4,6)
\put(24.45,8){\vector(-2,1){.01}}
\qbezier(17,0)(23.2,2.5)(29.4,5)
\put(25.2,3.3){\vector(2,1){.01}}
\end{picture}
\caption{An arc-coloured bipartite tournament for subcase 2.1 of the proof of Theorem \ref{z3.2}.}
\end{center}
\end{figure}
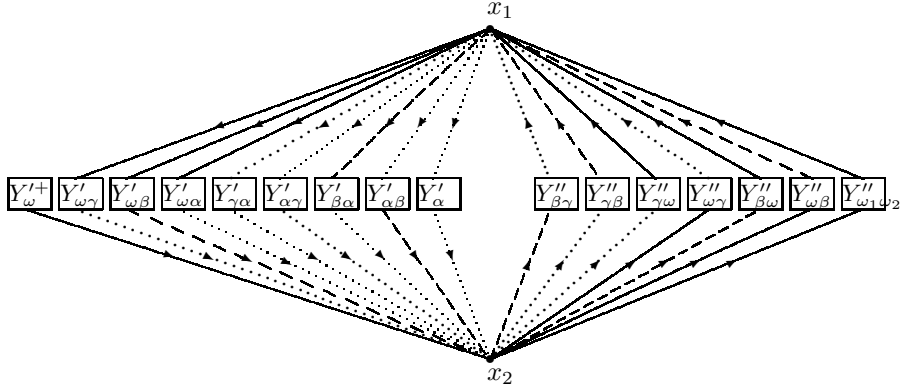

Note that $Y'_\alpha\ne \emptyset$ and $Y''_{\beta\gamma}\ne \emptyset$. Since every $4$-cycle is coloured with at least 3 colours, we have
$$Y'=Y'_{\alpha}\cup Y'_{\alpha\beta}\cup Y'_{\beta\alpha}\cup Y'_{\alpha\gamma}\cup Y'_{\gamma\alpha}\cup Y'_{\omega\alpha}\cup Y'_{\omega\beta}\cup Y'_{\omega\gamma}\cup Y'^{+}_{\omega},$$
$$Y''=Y''_{\beta\gamma}\cup Y''_{\gamma\beta}\cup Y''_{\beta\omega}\cup Y''_{\gamma\omega}\cup Y''_{\omega\beta}\cup Y''_{\omega\gamma}\cup Y''_{\omega_1\omega_2}.$$
For convenience, we denote the vertex in $Y'_{\alpha}\, (\mbox{resp. }Y'_{\alpha\beta},\, Y'_{\beta\alpha},\, Y'_{\alpha\gamma},\, Y'_{\gamma\alpha},\,Y'_{\omega\alpha},\, Y'_{\omega\beta},$
$Y'_{\omega\gamma},\, Y'^{+}_{\omega})$ by $y'_{\alpha}$ $(\mbox{resp. }\ y'_{\alpha\beta},\, y'_{\beta\alpha}, \, y'_{\alpha\gamma},\, y'_{\gamma\alpha},\,y'_{\omega\alpha},\, y'_{\omega\beta},\, y'_{\omega\gamma},\, y'^{+}_{\omega})$, the vertex in $Y''_{\beta\gamma}\,(\mbox{resp. } Y''_{\gamma\beta},\, Y''_{\beta\omega},\, Y''_{\gamma\omega},\, Y''_{\omega\beta},\, Y''_{\omega\gamma},\, Y''_{\omega_1\omega_2})$ by $y''_{\beta\gamma}\,(\mbox{resp. } y''_{\gamma\beta},\, y''_{\beta\omega},\, y''_{\gamma\omega},\,y''_{\omega\beta},\, y''_{\omega\gamma},$ $ y''_{\omega_1\omega_2})$.

If $Y''_{\omega_1\omega_2}\neq\emptyset$, then for any $y'\in Y'\setminus Y'^{+}_{\omega}$, $(y',x_2,y''_{\omega_1\omega_2},x_1)$ is a rainbow $(y',x_1)$-path; for any $y'^{+}_{\omega}\in Y'^{+}_{\omega}$, $(y'^{+}_{\omega},x_2,y''_{\beta\gamma},x_1)$ is a rainbow $(y'^{+}_{\omega},x_1)$-path. This implies that $\{x_1\}$ is a RP-kernel of $D$. So we assume $Y''_{\omega_1\omega_2}=\emptyset$.

If $Y''_{\omega\beta}\cup Y''_{\beta\omega}\neq\emptyset$ and $Y''_{\omega\gamma}\cup Y''_{\gamma\omega}\neq\emptyset$, then for any $y'\in Y'_{\alpha}\cup Y'_{\beta\alpha}\cup Y'_{\gamma\alpha}\cup Y'_{\omega\alpha}\cup Y'^{+}_{\omega}$, $(y',x_2,y''_{\beta\gamma},x_1)$ is a rainbow $(y',x_1)$-path; for any $y'\in Y'_{\alpha\beta}\cup Y'_{\omega\beta}$, $(y',x_2,y'',x_1)$ is a rainbow $(y',x_1)$-path, where $y''\in Y''_{\omega\gamma}\cup Y''_{\gamma\omega}$; for any $y'\in Y'_{\alpha\gamma}\cup Y'_{\omega\gamma}$, $(y',x_2,y'',x_1)$ is a rainbow $(y',x_1)$-path, where $y''\in Y''_{\omega\beta}\cup Y''_{\beta\omega}$. This implies that $\{x_1\}$ is a RP-kernel of $D$.

If exactly one of the subsets $Y''_{\omega\beta}\cup Y''_{\beta\omega}$ and $Y''_{\omega\gamma}\cup Y''_{\gamma\omega}$ is not empty set, w.l.o.g., assume $Y''_{\omega\beta}\cup Y''_{\beta\omega}\neq\emptyset$ and $Y''_{\omega\gamma}\cup Y''_{\gamma\omega}=\emptyset$. Now $Y''=Y''_{\beta\gamma}\cup Y''_{\gamma\beta}\cup Y''_{\beta\omega}\cup Y''_{\omega\beta}$.

For $Y'_{\omega\beta}=\emptyset$ and $Y'_{\alpha\beta}=\emptyset$, now $Y'=Y'_{\alpha}\cup Y'_{\beta\alpha}\cup Y'_{\alpha\gamma}\cup Y'_{\gamma\alpha}\cup Y'_{\omega\alpha}\cup Y'_{\omega\gamma}\cup Y'^{+}_{\omega}$.  For any $y'\in Y'_{\alpha}\cup Y'_{\beta\alpha}\cup Y'_{\gamma\alpha}\cup Y'_{\omega\alpha}\cup Y'^{+}_{\omega}$, $(y',x_2,y''_{\beta\gamma},x_1)$ is a rainbow $(y,x_1)$-path. For any $y'\in Y'_{\alpha\gamma}\cup Y'_{\omega\gamma}$, $(y',x_2,y'',x_1)$ is a rainbow $(y',x_1)$-path, where $y''\in Y''_{\omega\beta}\cup Y''_{\beta\omega}$. This implies that $\{x_1\}$ is a RP-kernel of $D$.

For $Y'_{\omega\beta}\neq\emptyset$ and $Y'_{\alpha\beta}\neq\emptyset$, now $Y'=Y'_{\alpha}\cup Y'_{\alpha\beta}\cup Y'_{\beta\alpha}\cup Y'_{\alpha\gamma}\cup Y'_{\gamma\alpha}\cup Y'_{\omega\alpha}\cup Y'_{\omega\beta}\cup Y'_{\omega\gamma}\cup Y'^{+}_{\omega}$. Note that there exists no rainbow path for any pair of vertices of $Y'_{\omega\beta}\cup Y'_{\alpha\beta}$ since $C(Y'_{\omega\beta}\cup Y'_{\alpha\beta},x_2)=\beta\in C(x_2,Y'')\cup C(Y'',x_1)$. Since every $4$-cycle is coloured with at least $3$ colours, we have $C(x_1,y'_{\omega\beta})\notin C(x_2,y''_{\omega\beta})\cup C(y''_{\beta\omega},x_1)$ for any $y'_{\omega\beta}\in Y'_{\omega\beta}$. For any $y'\in Y'_{\alpha}\cup Y'_{\beta\alpha}\cup Y'_{\alpha\gamma}\cup Y'_{\gamma\alpha}\cup Y'_{\omega\alpha}\cup Y'_{\omega\gamma}$, $(y',x_2,y'',x_1,y'_{\omega\beta})$ is a rainbow $(y',y'_{\omega\beta})$-path, where $y''\in Y''_{\omega\beta}\cup Y''_{\beta\omega}$. For any $y'^{+}_{\omega}\in Y'^{+}_{\omega}$, $(y'^{+}_{\omega},x_2,y''_{\beta\gamma},x_1,y'_{\alpha\beta})$ is a rainbow $(y'^{+}_{\omega},y'_{\alpha\beta})$-path. This implies that $Y'_{\omega\beta}\cup Y'_{\alpha\beta}$ is a RP-kernel of $D$.

For $Y'_{\omega\beta}\neq\emptyset$ and $Y'_{\alpha\beta}=\emptyset$, now $Y'=Y'_{\alpha}\cup Y'_{\beta\alpha}\cup Y'_{\alpha\gamma}\cup Y'_{\gamma\alpha}\cup Y'_{\omega\alpha}\cup Y'_{\omega\beta}\cup Y'_{\omega\gamma}\cup Y'^{+}_{\omega}$. By the proof above, there exists no rainbow path for any pair of vertices of $Y'_{\omega\beta}$ and each vertex of $ Y'_{\alpha}\cup Y'_{\beta\alpha}\cup Y'_{\alpha\gamma}\cup Y'_{\gamma\alpha}\cup Y'_{\omega\alpha}\cup Y'_{\omega\gamma}$ can reach $Y'_{\omega\beta}$ by a rainbow path passing through a vertex of $Y''_{\omega\beta}\cup Y''_{\beta\omega}$. If $|C(x_1,Y'_{\omega\beta})|\geqslant 2$, let $y'_{\omega\beta1}, y'_{\omega\beta2}\in Y'_{\omega\beta}$ with $C(x_1,y'_{\omega\beta1})\neq C(x_1,y'_{\omega\beta2})$. Let $y'^{+}_{\omega}\in Y'^{+}_{\omega}$ be arbitrary. Note that either $(y'^{+}_{\omega},x_2,y''_{\beta\gamma},x_1,y'_{\omega\beta1})$ or $(y'^{+}_{\omega},x_2,y''_{\beta\gamma},x_1,y'_{\omega\beta2})$ is a rainbow $(y'^{+}_{\omega},Y'_{\omega\beta})$-path. This implies that $Y'_{\omega\beta}$ is a RP-kernel of $D$. If $|C(x_1,Y'_{\omega\beta})|=1$,  let $U=\{y'^{+}_{\omega}\in Y'^{+}_{\omega}\,|\,C(y'^{+}_{\omega},x_2)=C(x_1,Y'_{\omega\beta})\}$. For any $y'\in Y'^{+}_{\omega}\setminus U$, $C(y',x_2)\neq C(x_1,Y'_{\omega\beta})$ and $(y',x_2,y''_{\beta\gamma},x_1,y'_{\omega\beta})$ is a rainbow path from $Y'^{+}_{\omega}\setminus U$ to $Y'_{\omega\beta}$. If there exists a vertex $u\in U$ with $C(x_1,u)\neq \beta$, for any $u'\in U\setminus \{u\}$, either $(u',x_2,y''_1,x_1,u)$ or $(u',x_2,y''_2,x_1,u)$ is a rainbow $(u',u)$-path, where $y''_1\in Y''_{\omega\beta}\cup Y''_{\beta\omega}$ and $y''_2\in Y''_{\beta\gamma}\cup Y''_{\gamma\beta}$.  Note that there exists no rainbow path for any pair of vertices of $Y'_{\omega\beta}\cup \{u\}$. This implies that $Y'_{\omega\beta}\cup \{u\}$ is a RP-kernel of $D$. If $C(x_1,U)= \beta$, then there exists no rainbow path for any pair of vertices of $Y'_{\omega\beta}\cup U$. This implies that $Y'_{\omega\beta}\cup U$ is a RP-kernel of $D$.

For $Y'_{\omega\beta}=\emptyset$ and $Y'_{\alpha\beta}\neq\emptyset$, now $Y'=Y'_{\alpha}\cup Y'_{\alpha\beta}\cup Y'_{\beta\alpha}\cup Y'_{\alpha\gamma}\cup Y'_{\gamma\alpha}\cup Y'_{\omega\alpha}\cup Y'_{\omega\gamma}\cup Y'^{+}_{\omega}$.
If $Y'_{\omega\alpha}\neq\emptyset$ or $Y'_{\gamma\alpha}\neq \emptyset$, then for any $y''\in Y''_{\gamma\beta}\cup Y''_{\omega\beta}$, $(y'',x_1,y',x_2)$ is a rainbow $(y'',x_2)$-path, where $y'\in Y'_{\gamma\alpha}\cup Y'_{\omega\alpha}$;  for any $y''\in Y''_{\beta\gamma}\cup Y''_{\beta\omega}$, $(y'',x_1,y'_{\alpha\beta},x_2)$ is a rainbow $(y'',x_2)$-path. This implies that $\{x_2\}$ is a RP-kernel of $D$.
If $Y'_{\omega\alpha}=\emptyset$ and $Y'_{\gamma\alpha}=\emptyset$, now $Y'=Y'_{\alpha}\cup Y'_{\alpha\beta}\cup Y'_{\beta\alpha}\cup Y'_{\alpha\gamma}\cup Y'_{\omega\gamma}\cup Y'^{+}_{\omega}$.
Note that there exists no rainbow path for any pair of vertices of $Y'_{\alpha}\cup Y'_{\alpha\beta}\cup Y'_{\beta\alpha}$. For any $y'\in Y'_{\alpha\gamma}\cup Y'_{\omega\gamma}$, $(y',x_2,y'',x_1,y'_{\alpha\beta})$ is a rainbow $(y',Y'_{\alpha\beta})$-path, where $y''\in Y''_{\omega\beta}\cup Y''_{\beta\omega}$. For any $y'^{+}_{\omega}\in Y'^{+}_{\omega}$, $(y'^{+}_{\omega},x_2,y''_{\beta\gamma},x_1,y'_{\alpha\beta})$ is a rainbow $(y'^{+}_{\omega},Y'_{\alpha\beta})$-path. This implies that $Y'_{\alpha}\cup Y'_{\alpha\beta}\cup Y'_{\beta\alpha}$ is a RP-kernel of $D$.

If $Y''_{\omega\beta}\cup Y''_{\beta\omega}=\emptyset$ and $Y''_{\omega\gamma}\cup Y''_{\gamma\omega}=\emptyset$, then $Y''=Y''_{\beta\gamma}\cup Y''_{\gamma\beta}$. This implies that there exists no rainbow path for any pair of vertices of $Y'_{\omega\beta}\cup Y'_{\omega\gamma}\cup Y'_{\alpha\gamma}\cup Y'_{\alpha\beta}$.

For $Y'_{\alpha\gamma}\cup Y'_{\alpha\beta}\neq \emptyset$ and $Y'_{\omega\beta}\cup Y'_{\omega\gamma}\neq\emptyset$, now $Y'=Y'_{\alpha}\cup Y'_{\alpha\beta}\cup Y'_{\beta\alpha}\cup Y'_{\alpha\gamma}\cup Y'_{\gamma\alpha}\cup Y'_{\omega\alpha}\cup Y'_{\omega\beta}\cup Y'_{\omega\gamma}\cup Y'^{+}_{\omega}$. For any $y'\in Y'\setminus Y'^{+}_{\omega}$, $(y',x_2,y''_{\beta\gamma},x_1,y'_{\omega\beta})$ and $(y',x_2,y''_{\beta\gamma},x_1,y'_{\omega\gamma})$ are rainbow $(y',Y'_{\omega\beta}\cup Y'_{\omega\gamma})$-paths; for any $y'^{+}_{\omega}\in Y'^{+}_{\omega}$, $(y'^{+}_{\omega},x_2,y''_{\beta\gamma},x_1,y'_{\alpha\gamma})$ and $(y'^{+}_{\omega},x_2,y''_{\beta\gamma},x_1,y'_{\alpha\beta})$ are rainbow $(y',Y'_{\alpha\gamma}\cup Y'_{\alpha\beta})$-paths. This implies that $Y'_{\omega\beta}\cup Y'_{\omega\gamma}\cup Y'_{\alpha\gamma}\cup Y'_{\alpha\beta}$ is a RP-kernel of $D$.

For $Y'_{\alpha\gamma}\cup Y'_{\alpha\beta}=\emptyset$ and $Y'_{\omega\beta}\cup Y'_{\omega\gamma}=\emptyset$, now $Y'=Y'_{\alpha}\cup Y'_{\beta\alpha}\cup Y'_{\gamma\alpha}\cup Y'_{\omega\alpha}\cup Y'^{+}_{\omega}$. For any $y'\in Y'$, $(y',x_2,y''_{\beta\gamma},x_1)$ is a rainbow $(y',x_1)$-path. This implies that $\{x_1\}$ is a RP-kernel of $D$.

For $Y'_{\alpha\gamma}\cup Y'_{\alpha\beta}\neq \emptyset$ and $Y'_{\omega\beta}\cup Y'_{\omega\gamma}=\emptyset$, now $Y'=Y'_{\alpha}\cup Y'_{\alpha\beta}\cup Y'_{\beta\alpha}\cup Y'_{\alpha\gamma}\cup Y'_{\gamma\alpha}\cup Y'_{\omega\alpha}\cup Y'^{+}_{\omega}$. If $Y'_{\omega\alpha}=\emptyset$, then $Y'=Y'_{\alpha}\cup Y'_{\alpha\beta}\cup Y'_{\beta\alpha}\cup Y'_{\alpha\gamma}\cup Y'_{\gamma\alpha}\cup  Y'^{+}_{\omega}$. Note that there exists no rainbow path for any pair of vertices of $Y'_{\alpha}\cup Y'_{\alpha\beta}\cup Y'_{\beta\alpha}\cup Y'_{\alpha\gamma}\cup Y'_{\gamma\alpha}$. For any $y'^{+}_{\omega}\in Y'^{+}_{\omega}$, $(y'^{+}_{\omega},x_2,y''_{\beta\gamma},x_1,y'_{\alpha})$ is a rainbow $(y'^{+}_{\omega},y'_{\alpha})$-path. This implies that $Y'_{\alpha}\cup Y'_{\alpha\beta}\cup Y'_{\beta\alpha}\cup Y'_{\alpha\gamma}\cup Y'_{\gamma\alpha}$ is a RP-kernel of $D$. If $Y'_{\omega\alpha}\neq\emptyset$, for any $y''\in Y''$, $(y'',x_1,y'_{\omega\alpha},x_2)$ is a rainbow $(y'',x_2)$-path. This implies that $\{x_2\}$ is a RP-kernel of $D$.

For $Y'_{\alpha\gamma}\cup Y'_{\alpha\beta}=\emptyset$ and $Y'_{\omega\beta}\cup Y'_{\omega\gamma}\neq\emptyset$, now $Y'=Y'_{\alpha}\cup Y'_{\beta\alpha}\cup Y'_{\gamma\alpha}\cup Y'_{\omega\alpha}\cup Y'_{\omega\beta}\cup Y'_{\omega\gamma}\cup Y'^{+}_{\omega}$. For any $y'\in Y'_{\alpha}\cup Y'_{\beta\alpha}\cup Y'_{\gamma\alpha}\cup Y'_{\omega\alpha}$, $(y',x_2,y''_{\beta\gamma},x_1,y'_{\omega\beta})$ and $(y',x_2,y''_{\beta\gamma},x_1,y'_{\omega\gamma})$ are rainbow $(y',Y'_{\omega\beta}\cup Y'_{\omega\gamma})$-paths.
If $|C(x_1,Y'_{\omega\beta}\cup Y'_{\omega\gamma})|\geqslant2$, let $y'_{1}, y'_{2}\in Y'_{\omega\beta}\cup Y'_{\omega\gamma}$ with $C(x_1,y'_{1})\neq C(x_1,y'_{2})$. Let $y'^{+}_{\omega}\in Y'^{+}_{\omega}$ be arbitrary. Note that either $(y'^{+}_{\omega},x_2,y''_{\beta\gamma},x_1,y'_{1})$ or $(y'^{+}_{\omega},x_2,y''_{\beta\gamma},x_1,y'_{2})$ is a rainbow $(y'^{+}_{\omega},Y'_{\omega\beta}\cup Y'_{\omega\gamma})$-path. This implies that $Y'_{\omega\beta}\cup Y'_{\omega\gamma}$ is a RP-kernel of $D$. If $|C(x_1,Y'_{\omega\beta}\cup Y'_{\omega\gamma})|=1$, let $U=\{y'^{+}_{\omega}\in Y'^{+}_{\omega}\,|\,C(y'^{+}_{\omega},x_2)=C(x_1,Y'_{\omega\beta}\cup Y'_{\omega\gamma})\}$. Then for any $y'\in Y'^{+}_{\omega}\setminus U$, we have $C(y',x_2)\neq C(x_1,Y'_{\omega\beta}\cup Y'_{\omega\gamma})$. Note that $(y',x_2,y''_{\beta\gamma},x_1,y'_{\omega\beta})$ and $(y',x_2,y''_{\beta\gamma},x_1,y'_{\omega\gamma})$ are two rainbow $(y',Y'_{\omega\beta}\cup Y'_{\omega\gamma})$-paths. If there exists a vertex $u\in U$ with $C(x_1,u)\notin \{\beta,\gamma,C(x_1,Y'_{\omega\beta}\cup Y'_{\omega\gamma})\}$, then for any $u'\in U\setminus \{u\}$, $(u',x_2,y''_{\beta\gamma},x_1,u)$ is a rainbow $(u',u)$-path. This implies that $Y'_{\omega\beta}\cup Y'_{\omega\gamma}\cup \{u\}$ is a RP-kernel of $D$. If $C(x_1,U)\subseteq \{\beta,\gamma,C(x_1,Y'_{\omega\beta}\cup Y'_{\omega\gamma})\}$, there exists no rainbow path for any pair of vertices of $Y'_{\omega\beta}\cup Y'_{\omega\gamma}\cup U$. This implies that $Y'_{\omega\beta}\cup Y'_{\omega\gamma}\cup U$ is a RP-kernel of $D$.

\vskip 2mm
\noindent \emph{Subcase 2.2.} $Y'\setminus Y^* =Y''\setminus Y^{**}=\emptyset$.
\vskip 2mm

Now $Y'=Y^*$, $Y''=Y^{**}$. This means $C(x_1,y')\neq C(y',x_2)$ and $C(x_2,y'')\neq C(y'',x_1)$ for any $y'\in Y'$ and $y''\in Y''$. Let $y'\in Y'$ be arbitrary. Assume $C(x_1,y')=\alpha$ and $C(y',x_2)=\beta$.

Define the following vertex subsets, see Figure 3 in which a box represent a set of vertices, and dotted arcs, dashed arcs, solid arcs represent respectively the arcs coloured by $\alpha$, $\beta$ and a colour not in $\{\alpha,\beta\}$.
$$Y'_{\alpha\beta}=\{y'\in Y'\,|\,C(x_1,y')=\alpha,\, C(y',x_2)=\beta\},$$
$$Y'_{\beta\alpha}=\{y'\in Y'\,|\,C(x_1,y')=\beta,\, C(y',x_2)=\alpha\},$$
$$Y'_{\alpha\omega}=\{y'\in Y'\,|\,C(x_1,y')=\alpha, \,C(y',x_2)\notin\{\alpha,\beta\}\},$$
$$Y'_{\beta\omega}=\{y'\in Y'\,|\,C(x_1,y')=\beta,\, C(y',x_2)\notin\{\alpha,\beta\}\},$$
$$Y'_{\omega\alpha}=\{y'\in Y'\,|\,C(x_1,y')\notin\{\alpha,\beta\},\, C(y',x_2)=\alpha\},$$
$$Y'_{\omega\beta}=\{y'\in Y'\,|\,C(x_1,y')\notin\{\alpha,\beta\},\, C(y',x_2)=\beta\},$$
$$Y'_{\omega_1\omega_2}=\{y'\in Y'\,|\,C(x_1,y'), C(y',x_2)\notin\{\alpha,\beta\}\mbox{ and }C(x_1,y')\neq C(y',x_2)\}.$$

Since every $4$-cycle is coloured with at least $3$ colours and $Y'_{\alpha\beta}\neq\emptyset$, $Y''$ can be divided into the following vertex subsets.
$$Y''_{\alpha\omega}=\{y''\in Y''\,|\,C(x_2,y'')=\alpha, \,C(y'',x_1)\notin\{\alpha,\beta\}\},$$
$$Y''_{\beta\omega}=\{y''\in Y''\,|\,C(x_2,y'')=\beta, \,C(y'',x_1)\notin\{\alpha,\beta\}\},$$
$$Y''_{\omega\alpha}=\{y''\in Y''\,|\,C(x_2,y'')\notin\{\alpha,\beta\},\, C(y'',x_1)=\alpha\},$$
$$Y''_{\omega\beta}=\{y''\in Y''\,|\,C(x_2,y'')\notin\{\alpha,\beta\},\, C(y'',x_1)=\beta\},$$
$$Y''_{\omega_1\omega_2}=\{y''\in Y''\,|\,C(x_2,y''), C(y'',x_1)\notin\{\alpha,\beta\}\mbox{ and }C(x_2,y'')\neq C(y'',x_1)\}.$$
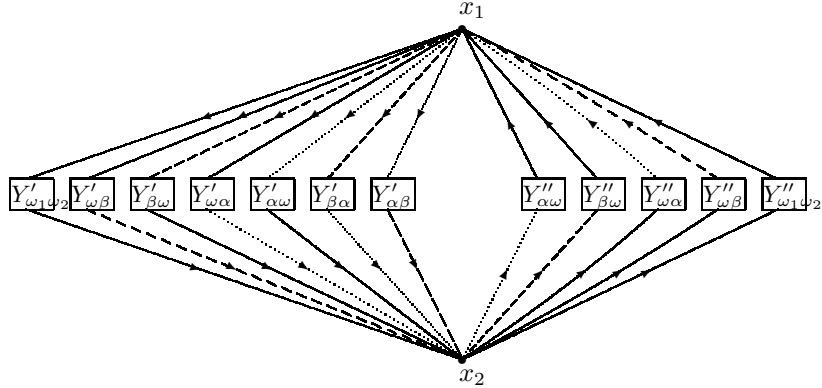
\begin {figure}[h]
\unitlength0.4cm
\begin{center}
\begin{picture}(28,12)
\put(16,11){\circle*{.3}}
\put(15.9,11.5){$x_1$}
\put(16,0){\circle*{.3}}
\put(15.9,-0.7){$x_2$}
\put(1,5){\framebox(1.4,1)}
\put(1,5.25){\small{$Y'_{\omega_1\omega_2}$}}
\put(3,5){\framebox(1.4,1)}
\put(3,5.25){\small{$Y'_{\omega\beta}$}}
\put(5,5){\framebox(1.4,1)}
\put(5,5.25){\small{$Y'_{\beta\omega}$}}
\put(7,5){\framebox(1.4,1)}
\put(7,5.25){\small{$Y'_{\omega\alpha}$}}
\put(9,5){\framebox(1.4,1)}
\put(9,5.25){\small{$Y'_{\alpha\omega}$}}
\put(11,5){\framebox(1.4,1)}
\put(11,5.25){\small{$Y'_{\beta\alpha}$}}
\put(13,5){\framebox(1.4,1)}
\put(13,5.25){\small{$Y'_{\alpha\beta}$}}
\qbezier(16,11)(8.75,8.5)(1.5,6)
\put(7.3,8){\vector(-3,-1){.01}}
\qbezier(16,0)(8.75,2.5)(1.5,5)
\put(7.3,3){\vector(3,-1){.01}}
\qbezier(16,11)(9.75,8.5)(3.5,6)
\put(8.5,8){\vector(-2,-1){.01}}
\multiput(3.5,5)(.02,-.01){15}{\line(1,0){.037182}}
\multiput(4,4.8)(.02,-.01){15}{\line(1,0){.037182}}
\multiput(4.5,4.6)(.02,-.01){15}{\line(1,0){.037182}}
\multiput(5,4.4)(.02,-.01){15}{\line(1,0){.037182}}
\multiput(5.5,4.2)(.02,-.01){15}{\line(1,0){.037182}}
\multiput(6,4)(.02,-.01){15}{\line(1,0){.037182}}
\multiput(6.5,3.8)(.02,-.01){15}{\line(1,0){.037182}}
\multiput(7,3.6)(.02,-.01){15}{\line(1,0){.037182}}
\multiput(7.5,3.4)(.02,-.01){15}{\line(1,0){.037182}}
\multiput(8,3.2)(.02,-.01){15}{\line(1,0){.037182}}
\multiput(8.5,3)(.02,-.01){15}{\line(1,0){.037182}}
\multiput(9,2.8)(.02,-.01){15}{\line(1,0){.037182}}
\multiput(9.5,2.6)(.02,-.01){15}{\line(1,0){.037182}}
\multiput(10,2.4)(.02,-.01){15}{\line(1,0){.037182}}
\multiput(10.5,2.2)(.02,-.01){15}{\line(1,0){.037182}}
\multiput(11,2)(.02,-.01){15}{\line(1,0){.037182}}
\multiput(11.5,1.8)(.02,-.01){15}{\line(1,0){.037182}}
\multiput(12,1.6)(.02,-.01){15}{\line(1,0){.037182}}
\multiput(12.5,1.4)(.02,-.01){15}{\line(1,0){.037182}}
\multiput(13,1.2)(.02,-.01){15}{\line(1,0){.037182}}
\multiput(13.5,1)(.02,-.01){15}{\line(1,0){.037182}}
\multiput(14,0.8)(.02,-.01){15}{\line(1,0){.037182}}
\multiput(14.5,0.6)(.02,-.01){15}{\line(1,0){.037182}}
\multiput(15,0.4)(.02,-.01){15}{\line(1,0){.037182}}
\multiput(15.5,0.2)(.02,-.01){15}{\line(1,0){.037182}}

\put(8.5,3){\vector(2,-1){.01}}
\multiput(16,11)(-.02,-.01){15}{\line(1,0){.037182}}
\multiput(15.55,10.8)(-.02,-.01){15}{\line(1,0){.037182}}
\multiput(15.13,10.6)(-.02,-.01){15}{\line(1,0){.037182}}
\multiput(14.7,10.4)(-.02,-.01){15}{\line(1,0){.037182}}
\multiput(14.29,10.2)(-.02,-.01){15}{\line(1,0){.037182}}
\multiput(13.87,10)(-.02,-.01){15}{\line(1,0){.037182}}
\multiput(13.45,9.8)(-.02,-.01){15}{\line(1,0){.037182}}
\multiput(13.03,9.6)(-.02,-.01){15}{\line(1,0){.037182}}
\multiput(12.61,9.4)(-.02,-.01){15}{\line(1,0){.037182}}
\multiput(12.19,9.2)(-.02,-.01){15}{\line(1,0){.037182}}
\multiput(11.77,9)(-.02,-.01){15}{\line(1,0){.037182}}
\multiput(11.35,8.8)(-.02,-.01){15}{\line(1,0){.037182}}
\multiput(10.93,8.6)(-.02,-.01){15}{\line(1,0){.037182}}
\multiput(10.51,8.4)(-.02,-.01){15}{\line(1,0){.037182}}
\multiput(10.09,8.2)(-.02,-.01){15}{\line(1,0){.037182}}
\multiput(9.67,8)(-.02,-.01){15}{\line(1,0){.037182}}
\multiput(9.25,7.8)(-.02,-.01){15}{\line(1,0){.037182}}
\multiput(8.83,7.6)(-.02,-.01){15}{\line(1,0){.037182}}
\multiput(8.41,7.4)(-.02,-.01){15}{\line(1,0){.037182}}
\multiput(7.99,7.2)(-.02,-.01){15}{\line(1,0){.037182}}
\multiput(7.57,7)(-.02,-.01){15}{\line(1,0){.037182}}
\multiput(7.15,6.8)(-.02,-.01){15}{\line(1,0){.037182}}
\multiput(6.73,6.6)(-.02,-.01){15}{\line(1,0){.037182}}
\multiput(6.31,6.4)(-.02,-.01){15}{\line(1,0){.037182}}
\multiput(5.89,6.2)(-.02,-.01){15}{\line(1,0){.037182}}

\put(9.7,8){\vector(-2,-1){.01}}

\qbezier(16,0)(10.75,2.5)(5.5,5)
\put(9.7,3){\vector(2,-1){.01}}
\qbezier(16,11)(11.75,8.5)(7.5,6)
\put(10.9,8){\vector(-3,-2){.01}}
\qbezier[60](16,0)(11.75,2.5)(7.5,5)
\put(10.9,3){\vector(3,-2){.01}}
\qbezier[60](16,11)(12.75,8.5)(9.5,6)
\put(12.1,8){\vector(-4,-3){.01}}
\qbezier(16,0)(12.75,2.5)(9.5,5)
\put(12.1,3){\vector(4,-3){.01}}
\multiput(16,11)(-.01,-.01){20}{\line(1,0){.037182}}
\multiput(15.73,10.7)(-.01,-.01){20}{\line(1,0){.037182}}
\multiput(15.46,10.4)(-.01,-.01){20}{\line(1,0){.037182}}
\multiput(15.19,10.1)(-.01,-.01){20}{\line(1,0){.037182}}
\multiput(14.92,9.8)(-.01,-.01){20}{\line(1,0){.037182}}
\multiput(14.65,9.5)(-.01,-.01){20}{\line(1,0){.037182}}
\multiput(14.38,9.2)(-.01,-.01){20}{\line(1,0){.037182}}
\multiput(14.11,8.9)(-.01,-.01){20}{\line(1,0){.037182}}
\multiput(13.84,8.6)(-.01,-.01){20}{\line(1,0){.037182}}
\multiput(13.57,8.3)(-.01,-.01){20}{\line(1,0){.037182}}
\multiput(13.3,8)(-.01,-.01){20}{\line(1,0){.037182}}
\multiput(13.03,7.7)(-.01,-.01){20}{\line(1,0){.037182}}
\multiput(12.76,7.4)(-.01,-.01){20}{\line(1,0){.037182}}
\multiput(12.49,7.1)(-.01,-.01){20}{\line(1,0){.037182}}
\multiput(12.22,6.8)(-.01,-.01){20}{\line(1,0){.037182}}
\multiput(11.95,6.5)(-.01,-.01){20}{\line(1,0){.037182}}
\multiput(11.68,6.2)(-.01,-.01){20}{\line(1,0){.037182}}

\put(13.3,8){\vector(-1,-1){.01}}
\qbezier[60](16,0)(13.75,2.5)(11.5,5)
\put(13.3,3){\vector(1,-1){.01}}
\qbezier[60](16,11)(14.75,8.5)(13.5,6)
\put(14.5,8){\vector(-1,-2){.01}}
\multiput(13.5,5)(.01,-.02){20}{\line(1,0){.037182}}
\multiput(13.75,4.5)(.01,-.02){20}{\line(1,0){.037182}}
\multiput(14,4)(.01,-.02){20}{\line(1,0){.037182}}
\multiput(14.25,3.5)(.01,-.02){20}{\line(1,0){.037182}}
\multiput(14.5,3)(.01,-.02){20}{\line(1,0){.037182}}
\multiput(14.75,2.5)(.01,-.02){20}{\line(1,0){.037182}}
\multiput(15,2)(.01,-.02){20}{\line(1,0){.037182}}
\multiput(15.25,1.5)(.01,-.02){20}{\line(1,0){.037182}}
\multiput(15.5,1)(.01,-.02){20}{\line(1,0){.037182}}
\multiput(15.75,0.5)(.01,-.02){20}{\line(1,0){.037182}}

\put(14.5,3){\vector(1,-2){.01}}

\put(18,5){\framebox(1.4,1)}
\put(18,5.25){\small{$Y''_{\alpha\omega}$}}
\put(20,5){\framebox(1.4,1)}
\put(20,5.25){\small{$Y''_{\beta\omega}$}}
\put(22,5){\framebox(1.4,1)}
\put(22,5.25){\small{$Y''_{\omega\alpha}$}}
\put(24,5){\framebox(1.4,1)}
\put(24,5.25){\small{$Y''_{\omega\beta}$}}
\put(26,5){\framebox(1.4,1)}
\put(26,5.25){\small{$Y''_{\omega_1\omega_2}$}}
\qbezier(16,11)(21.25,8.5)(26.5,6)
\put(22.3,8){\vector(-2,1){.01}}
\qbezier(16,0)(21.25,2.5)(26.5,5)
\put(22.3,3){\vector(2,1){.01}}
\qbezier(16,11)(17.25,8.5)(18.5,6)
\put(17.5,8){\vector(-1,2){.01}}
\qbezier[40](16,0)(17.25,2.5)(18.5,5)
\put(17.5,3){\vector(1,2){.01}}
\qbezier(16,11)(18.25,8.5)(20.5,6)
\put(18.7,8){\vector(-1,1){.01}}
\multiput(16,0)(.01,.01){20}{\line(1,0){.037182}}
\multiput(16.25,0.3)(.01,.01){20}{\line(1,0){.037182}}
\multiput(16.52,0.6)(.01,.01){20}{\line(1,0){.037182}}
\multiput(16.79,0.9)(.01,.01){20}{\line(1,0){.037182}}
\multiput(17.06,1.2)(.01,.01){20}{\line(1,0){.037182}}
\multiput(17.33,1.5)(.01,.01){20}{\line(1,0){.037182}}
\multiput(17.6,1.8)(.01,.01){20}{\line(1,0){.037182}}
\multiput(17.87,2.1)(.01,.01){20}{\line(1,0){.037182}}
\multiput(18.14,2.4)(.01,.01){20}{\line(1,0){.037182}}
\multiput(18.41,2.7)(.01,.01){20}{\line(1,0){.037182}}
\multiput(18.68,3)(.01,.01){20}{\line(1,0){.037182}}
\multiput(18.95,3.3)(.01,.01){20}{\line(1,0){.037182}}
\multiput(19.22,3.6)(.01,.01){20}{\line(1,0){.037182}}
\multiput(19.49,3.9)(.01,.01){20}{\line(1,0){.037182}}
\multiput(19.76,4.2)(.01,.01){20}{\line(1,0){.037182}}
\multiput(20.03,4.5)(.01,.01){20}{\line(1,0){.037182}}
\multiput(20.3,4.8)(.01,.01){20}{\line(1,0){.037182}}

\put(18.7,3){\vector(1,1){.01}}
\qbezier[50](16,11)(19.75,8.5)(22.5,6)
\put(20.15,8){\vector(-4,3){.01}}
\qbezier(16,0)(19.75,2.5)(22.5,5)
\put(20.15,3){\vector(4,3){.01}}
\multiput(24.5,6)(-.03,.02){10}{\line(1,0){.05}}
\multiput(24.05,6.3)(-.03,.02){10}{\line(1,0){.05}}
\multiput(23.55,6.6)(-.03,.02){10}{\line(1,0){.05}}
\multiput(23.05,6.9)(-.03,.02){10}{\line(1,0){.05}}
\multiput(22.6,7.2)(-.03,.02){10}{\line(1,0){.05}}
\multiput(22.1,7.5)(-.03,.02){10}{\line(1,0){.05}}
\multiput(21.6,7.8)(-.03,.02){10}{\line(1,0){.05}}
\multiput(21.1,8.1)(-.03,.02){10}{\line(1,0){.05}}
\multiput(20.6,8.4)(-.03,.02){10}{\line(1,0){.05}}
\multiput(20.1,8.7)(-.03,.02){10}{\line(1,0){.05}}
\multiput(19.6,9)(-.03,.02){10}{\line(1,0){.05}}
\multiput(19.1,9.3)(-.03,.02){10}{\line(1,0){.05}}
\multiput(18.55,9.6)(-.03,.02){10}{\line(1,0){.05}}
\multiput(18,9.9)(-.03,.02){10}{\line(1,0){.05}}
\multiput(17.45,10.2)(-.03,.02){10}{\line(1,0){.05}}
\multiput(16.9,10.5)(-.03,.02){10}{\line(1,0){.05}}
\multiput(16.3,10.8)(-.03,.02){10}{\line(1,0){.05}}
\put(21.35,8){\vector(-3,2){.01}}
\qbezier(16,0)(20.75,2.5)(24.5,5)
\put(21.35,3){\vector(3,2){.01}}
\end{picture}
\caption{An arc-coloured bipartite tournament for subcase 2.2 of the proof of Theorem \ref{z3.2}.}
\end{center}
\end{figure}

Now,
$$Y'=Y'_{\alpha\beta}\cup Y'_{\beta\alpha}\cup Y'_{\alpha\omega}\cup Y'_{\beta\omega}\cup Y'_{\omega\alpha}\cup Y'_{\omega\beta}\cup Y'_{\omega_1\omega_2},$$
$$Y''=Y''_{\alpha\omega}\cup Y''_{\beta\omega}\cup Y''_{\omega\alpha}\cup Y''_{\omega\beta}\cup Y''_{\omega_1\omega_2}.$$
For convenience, we denote the vertex in $Y'_{\alpha\beta}\,(\mbox{resp. } Y'_{\beta\alpha},\, Y'_{\alpha\omega},\, Y'_{\beta\omega},\, Y'_{\omega\alpha},\, Y'_{\omega\beta},\,$ $ Y'_{\omega_1\omega_2})$ by $y'_{\alpha\beta}(\mbox{resp. } y'_{\beta\alpha},\, y'_{\alpha\omega},\, y'_{\beta\omega},\, y'_{\omega\alpha},\, y'_{\omega\beta},\, y'_{\omega_1\omega_2})$, the vertex in $Y''_{\alpha\omega}\,(\mbox{resp. }$ $Y''_{\beta\omega},\,Y''_{\omega\alpha},\, Y''_{\omega\beta},$ $Y''_{\omega_1\omega_2}$) by $y''_{\alpha\omega}$ (resp. $y''_{\beta\omega},$ $y''_{\omega\alpha}$, $y''_{\omega\beta}$, $y''_{\omega_1\omega_2}$).

If $Y'_{\omega_1\omega_2}\neq\emptyset$ or $Y'_{\omega_1\omega_2}\cup Y''_{\omega\alpha}\cup Y''_{\omega\beta}=\emptyset$, then for any $y''\in Y''_{\omega\alpha}\cup Y''_{\omega\beta}$, $(y'',x_1,y'_{\omega_1\omega_2},x_2)$ is a rainbow $(y'',x_2)$-path; for any $y''\in Y''_{\alpha\omega}\cup Y''_{\beta\omega}\cup Y''_{\omega_1\omega_2}$, $(y'',x_1,y'_{\alpha\beta},x_2)$ is a rainbow $(y'',x_2)$-path. This implies that $\{x_2\}$ is a RP-kernel of $D$. So we assume $Y'_{\omega_1\omega_2}=\emptyset$ and $Y''_{\omega\alpha}\cup Y''_{\omega\beta}\neq \emptyset$.

If $Y''_{\omega\alpha}\neq\emptyset$ and $Y''_{\omega\beta}\neq\emptyset$, then $Y''=Y''_{\alpha\omega}\cup Y''_{\beta\omega}\cup Y''_{\omega\beta}\cup Y''_{\omega\alpha}\cup Y''_{\omega_1\omega_2}$. For any $y'\in Y'_{\alpha\beta}\cup Y'_{\omega\beta}$, $(y',x_2,y''_{\omega\alpha},x_1)$ is a rainbow $(y',x_1)$-path. For any $y'\in Y'_{\beta\alpha}\cup Y'_{\omega\alpha}$, $(y',x_2,y''_{\omega\beta},x_1)$ is a rainbow $(y',x_1)$-path. Since every $4$-cycle is coloured with at least $3$ colours, we have $C(y'_{\alpha\omega},x_2)\notin C(x_2,Y''_{\omega\alpha})$ for any $y'_{\alpha\omega}\in Y'_{\alpha\omega}$ and $C(y'_{\beta\omega},x_2)\notin C(x_2,Y''_{\omega\beta})$ for any $y'_{\beta\omega}\in Y'_{\beta\omega}$. It follows that $(y'_{\alpha\omega},x_2,y''_{\omega\alpha},x_1)$ is a rainbow $(y'_{\alpha\omega},x_1)$-path and $(y'_{\beta\omega},x_2,y''_{\omega\beta},x_1)$ is a rainbow $(y'_{\beta\omega},x_1)$-path. This implies that $\{x_1\}$ is a RP-kernel of $D$.

If exactly one of the subsets $Y''_{\omega\alpha}$ and $Y''_{\omega\beta}$ is not empty set, w.l.o.g., assume $Y''_{\omega\alpha}\neq\emptyset$ and $Y''_{\omega\beta}=\emptyset$. Now, $Y''=Y''_{\alpha\omega}\cup Y''_{\beta\omega}\cup Y''_{\omega\alpha}\cup Y''_{\omega_1\omega_2}$.

For $Y'_{\beta\omega}\cup Y'_{\omega\beta}\neq\emptyset$, we see that for any $y''\in Y''_{\alpha\omega}\cup Y''_{\beta\omega}\cup Y''_{\omega_1\omega_2}$, $(y'',x_1,$ $y'_{\alpha\beta},$ $x_2)$ is a rainbow $(y'',x_2)$-path; for any $y''_{\omega\alpha}\in Y''_{\omega\alpha}$, $(y''_{\omega\alpha},x_1,y',x_2)$ is a rainbow $(y''_{\omega\alpha},x_2)$-path, where $y'\in Y'_{\beta\omega}\cup Y'_{\omega\beta}$. This implies that $\{x_2\}$ is a RP-kernel of $D$.

For $Y''_{\beta\omega}\neq\emptyset$, we see that for any $y'\in Y'_{\beta\alpha}\cup Y'_{\omega\alpha}\cup Y'_{\beta\omega}$, $(y',x_2,y''_{\beta\omega},x_1)$ is a rainbow $(y',x_1)$-path; for any $y'\in Y'_{\alpha\beta}\cup Y'_{\alpha\omega}\cup Y'_{\omega\beta}$, $(y',x_2,y''_{\omega\alpha},x_1)$ is a rainbow $(y',x_1)$-path. This implies that $\{x_1\}$ is a RP-kernel of $D$.

For $Y'_{\beta\omega}\cup Y'_{\omega\beta}=\emptyset$ and $Y''_{\beta\omega}=\emptyset$, we see that $Y'=Y'_{\alpha\beta}\cup Y'_{\beta\alpha}\cup Y'_{\alpha\omega}\cup Y'_{\omega\alpha}$ and $Y''=Y''_{\alpha\omega}\cup Y''_{\omega\alpha}\cup Y''_{\omega_1\omega_2}$. If $Y''_{\omega_1\omega_2}\neq\emptyset$, then for any $y'\in Y'\setminus Y'_{\alpha\omega}$, $(y',x_2,y''_{\omega_1\omega_2},x_1)$ is a rainbow $(y',x_1)$-path; for any $y'_{\alpha\omega}\in Y'_{\alpha\omega}$, $(y'_{\alpha\omega},x_2,y''_{\alpha\omega},x_1)$ is a rainbow $(y'_{\alpha\omega},x_1)$-path. This implies that $\{x_1\}$ is a RP-kernel of $D$. If $Y''_{\omega_1\omega_2}=\emptyset$, then $Y''=Y''_{\alpha\omega}\cup Y''_{\omega\alpha}$. For any $y'\in Y'_{\alpha\beta}\cup Y'_{\beta\alpha}\cup Y'_{\omega\alpha}$, $(y',x_2,y''_{\omega\alpha})$ is a rainbow $(y',y''_{\omega\alpha})$-path. By the proof above, for any $y'_{\alpha\omega}\in Y'_{\alpha\omega}$, $(y'_{\alpha\omega},x_2,y''_{\omega\alpha})$ is a rainbow $(y'_{\alpha\omega},y''_{\omega\alpha})$-path. If $|C(x_2,Y''_{\omega\alpha})|\geqslant 2$, let $y''_{\omega\alpha1},y''_{\omega\alpha2}\in Y''_{\omega\alpha}$ with $C(x_2,y''_{\omega\alpha1})\neq C(x_2,y''_{\omega\alpha2})$. For any $y''_{\alpha\omega}\in Y''_{\alpha\omega}$, either $(y''_{\alpha\omega},x_1,y'_{\alpha\beta},x_2,y''_{\omega\alpha1})$ or $(y''_{\alpha\omega},x_1,y'_{\alpha\beta},x_2,y''_{\omega\alpha2})$ is a rainbow $(y''_{\alpha\omega}, y''_{\omega\alpha})$-path. This implies that $Y''_{\omega\alpha}$ is a RP-kernel of $D$. If $|C(x_2,Y''_{\omega\alpha})|=1$, let $U=\{y''_{\alpha\omega}\in Y''_{\alpha\omega}\,|\,C(y''_{\alpha\omega},x_1)=C(x_2,Y''_{\omega\alpha})\}$. Note that there exists no rainbow path for any pair of vertices of $Y''_{\omega\alpha}\cup U$. For any $y''\in Y''_{\alpha\omega}\setminus U$, we have $C(y'',x_1)\neq C(x_2,Y''_{\omega\alpha})$ and $(y'',x_1,y'_{\alpha\beta},x_2,y''_{\omega\alpha})$ is a rainbow $(y'',y''_{\omega\alpha})$-path. This implies that $Y''_{\omega\alpha}\cup U$ is a RP-kernel of $D$.

In any case, we can find a RP-kernel of $D$. This proof is complete.
\end{proof}

By Theorem \ref{z3.2}, the following corollary is immediate.

\begin{corollary}
Let $D=(X,Y)$ be an $m$-arc-coloured bipartite tournament with $\min\{|X|,|Y|\}=2$. If every $4$-cycle contained in $D$ is rainbow, then $D$ has a RP-kernel.
\end{corollary}

\begin{remark}
The condition ``every $4$-cycle is coloured with at least $3$ colours" in Theorem \ref{z3.2} cannot be reduced. An arc-coloured bipartite tournament with $|X|=2$ shown in Figure $4$ satisfying ``every $4$-cycle is $2$-arc-coloured" has no RP-kernel, in which solid arcs and dotted arcs are represent arcs coloured by two distinct colours. Large $m$-arc-coloured bipartite tournaments with no RP-kernel can be obtained by adding new vertices to $Y$ and new colours such that these new vertices dominate completely $X$.
\end{remark}

\begin {figure}[h]
\unitlength0.4cm
\begin{center}
\begin{picture}(10,7.5)
\put(3,6){\circle*{.3}}
\put(2.9,6.5){$x_1$}
\put(1,1){\circle*{.3}}
\put(0.9,0){$y_1$}
\put(7,6){\circle*{.3}}
\put(6.9,6.5){$x_2$}
\put(5,1){\circle*{.3}}
\put(4.9,0){$y_2$}
\put(9,1){\circle*{.3}}
\put(8.9,0){$y_3$}
\qbezier(3,6)(4,3.5)(5,1)
\put(4.4,2.5){\vector(1,-2){.01}}
\qbezier(1,1)(2,3.5)(3,6)
\put(1.6,2.5){\vector(-1,-2){.01}}
\qbezier(7,6)(8,3.5)(9,1)
\put(8.4,2.5){\vector(1,-2){.01}}
\qbezier[60](5,1)(6,3.5)(7,6)
\put(6.4,4.5){\vector(1,2){.01}}
\qbezier[60](1,1)(4,3.5)(7,6)
\put(5.4,4.7){\vector(1,1){.01}}

\qbezier[60](9,1)(6,3.5)(3,6)
\put(4.6,4.7){\vector(-1,1){.01}}
\end{picture}
\caption{ An arc-coloured bipartite tournament with $|X|=2$ satisfying ``every $4$-cycle is $2$-arc-coloured" has no RP-kernel.}
\end{center}
\end{figure}
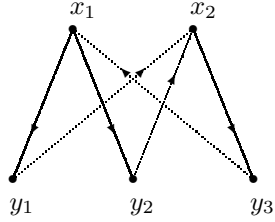

In the following proof, we consider $\min\{|X|,|Y|\}\geqslant3$.

\begin{lemma}\cite{7}\label{z3.3}
Let $D=(X,Y)$ be a bipartite tournament. Then the following statements hold:

(a) let $C=(u_0,u_1,u_2,\ldots,u_n)$ be a walk in $D$. For $\{i,j\}\subseteq \{1,2,\ldots,n\}$, $u_i,u_j$ are adjacent if and only if $j-i\equiv 1\,(\emph{mod } 2)$.

(b) every closed walk of length at most 6 is a cycle of $D$.
\end{lemma}

Let $CB_5$ be a bipartite tournament, which has $V(CB_5)=\{u_1,u_2,u_3,u_4,u_5\}$ and $A(CB_5)=\{(u_1,u_2),(u_2,u_3),(u_3,u_4),(u_4,u_5),(u_4,u_1),(u_5,u_2)\}$. Let $TB_4$ be a bipartite tournament, which has $V(TB_4)=\{u_1,u_2,u_3,u_4\}$ and $A(TB_4)=\{(u_1,u_2),(u_2,u_3),(u_3,u_4),(u_1,u_4)\}$. See Figure $5$.

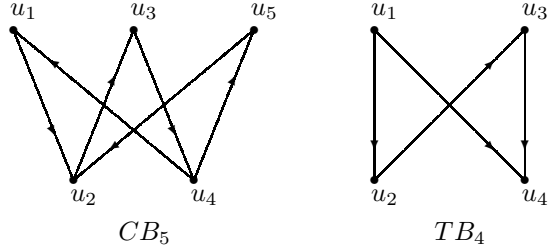
\begin {figure}[h]
\unitlength0.4cm
\begin{center}
\begin{picture}(19,8.5)
\put(1,7){\circle*{.3}}
\put(0.9,7.5){$u_1$}
\put(3,2){\circle*{.3}}
\put(2.9,1.3){$u_2$}
\put(5,7){\circle*{.3}}
\put(4.9,7.5){$u_3$}
\put(7,2){\circle*{.3}}
\put(6.9,1.3){$u_4$}
\put(9,7){\circle*{.3}}
\put(8.9,7.5){$u_5$}
\qbezier(1,7)(2,4.5)(3,2)
\put(2.4,3.5){\vector(1,-2){.01}}
\qbezier(3,2)(4,4.5)(5,7)
\put(4.4,5.5){\vector(1,2){.01}}
\qbezier(5,7)(6,4.5)(7,2)
\put(6.4,3.5){\vector(1,-2){.01}}
\qbezier(7,2)(8,4.5)(9,7)
\put(8.4,5.5){\vector(1,2){.01}}
\qbezier(9,7)(6,4.5)(3,2)
\put(2.18,6){\vector(-3,2){.01}}
\qbezier(7,2)(4,4.5)(1,7)
\put(4.18,3){\vector(-3,-2){.01}}

\put(4.5,0){$CB_5$}

\put(13,7){\circle*{.3}}
\put(12.9,7.5){$u_1$}
\put(13,2){\circle*{.3}}
\put(12.9,1.3){$u_2$}
\put(18,7){\circle*{.3}}
\put(17.9,7.5){$u_3$}
\put(18,2){\circle*{.3}}
\put(17.9,1.3){$u_4$}
\qbezier(13,7)(13,4.5)(13,2)
\put(13,3){\vector(0,-1){.01}}
\qbezier(13,2)(15.5,4.5)(18,7)
\put(17,6){\vector(1,1){.01}}
\qbezier(18,7)(18,4.5)(18,2)
\put(18,3){\vector(0,-1){.01}}
\qbezier(13,7)(15.5,4.5)(18,2)
\put(17,3){\vector(1,-1){.01}}

\put(15,0){$TB_4$}
\end{picture}
\caption{Two bipartite tournaments $CB_5$ and $TB_4$.}
\end{center}
\end{figure}
\begin{lemma}\label{z3.4}
Let $D=(X,Y)$ be an $m$-arc-coloured bipartite tournament with $\min\{|X|,|Y|\}\geqslant3$. If all $4$-cycles, $6$-cycles and induced subdigraphs $CB_5$ in $D$ are rainbow, and all induced subdigraphs $TB_4$ in $D$ are properly coloured, then for any pair of distinct vertices $u,v\in V(D)$ satisfying there exists a rainbow $(u,v)$-path and no rainbow $(v,u)$-path in $D$, at least one of the following conditions holds:

($a$) $u\rightarrow v$;

($b$) there exists a $(u,v)$-path of length $2$.
\end{lemma}

\begin{proof}
Let $P=(u=u_0,u_1,u_2,\ldots,u_n=v)$ be the shortest rainbow $(u,v)$-path in $D$. The result holds clearly for $n\leqslant2$. Now assume $n\geqslant3$.

If $n$ is odd, by Lemma \ref{z3.3} ($a$), we have $u_0,u_n$ are adjacent. Since there is no rainbow $(v,u)$-path in $D$, we have $u_0\rightarrow u_n$. The result holds. So we assume that $n$ is even.

Also by Lemma \ref{z3.3} ($a$), we have $u_1,u_n$ are adjacent. If $u_1\rightarrow u_n$, then $(u=u_0,u_1,u_n=v)$ is a $(u,v)$-path of length $2$ and the result holds. So we assume $u_n\rightarrow u_1$.

If $n=4$, then $P=(u=u_0,u_1,u_2,u_3,u_4=v)$. For $u_3\rightarrow u_0$, we see that $D[u_0,u_1,u_2,u_3,u_4]$ is an induced rainbow $CB_5$, which implies $(u_4,u_1,u_2,$ $u_3,u_0)$ is a rainbow $(v,u)$-path, a contradiction. For $u_0\rightarrow u_3$, we have $(u_0, u_3, u_4)$ is a $(u,v)$-path of length $2$. So we assume $n\geqslant6$.

If $u_0\rightarrow u_{i_0}\rightarrow u_n$ for some $i_0\in \{3,5,\ldots, n-1\}$, then $(u=u_0,u_{i_0},u_n=v)$ is a $(u,v)$-path of length $2$ and the result holds. So we assume that either $u_i\rightarrow u_{0}$ or $u_{n}\rightarrow u_i$ for each $i\in \{3,5,\ldots, n-1\}$.

\vskip 2mm
\noindent {\bf Claim 1.} For each $i\in \{3,5,\ldots, n-3\}$, $u_0\rightarrow u_{i}$ and $u_{n}\rightarrow u_i$.

\begin{proof}It is sufficient to show that $u_0\rightarrow u_i$. We process by induction on $i$. For $i=3$, suppose to the contrary that $u_3\rightarrow u_{0}$. If $u_3\rightarrow u_{n}$, then $D[u_0,u_1,u_2,u_3,u_n]$ is an induced rainbow $CB_5$. It follows that $(u_n,u_1,u_2,u_3,u_0)$ is a rainbow $(v,u)$-path in $D$, a contradiction. If $u_n\rightarrow u_{3}$, then $D[u_n,u_3,u_0,u_1]$ is $TB_4$ which is properly coloured. It follows that $(u_n,u_3,u_0)$ is a rainbow $(v,u)$-path in $D$, a contradiction. Thus $(u_0,u_{3})\in A(D)$.

Assume that the claim holds for $i<n-3$. We consider the case $i=n-3$.

Suppose to the contrary that $u_{n-3}\rightarrow u_{0}$. By the induction hypothesis, we have $u_0\rightarrow u_{n-5}$ and $u_{n}\rightarrow u_{n-5}$. If $u_{n-3}\rightarrow u_{n}$, then $D[u_0,u_{n-5},u_{n-4},u_{n-3},u_n]$ is an induced rainbow $CB_5$. It follows that $(u_n,u_{n-5},u_{n-4},u_{n-3},u_0)$ is a rainbow $(v,u)$-path in $D$, a contradiction. If $u_n\rightarrow u_{n-3}$, then $D[u_n,u_{n-3},u_0,u_{n-5}]$ is $TB_4$ which is properly coloured. It follows that $(u_n,u_{n-3},u_0)$ is a rainbow $(v,u)$-path in $D$, a contradiction. So $(u_0,u_{n-3})\in A(D)$.\end{proof}

Now we show $u_0\rightarrow u_{n-1}$. Suppose to the contrary that $u_{n-1}\rightarrow u_{0}$. By Claim 1, we have $u_0\rightarrow u_{n-3}$ and $u_{n}\rightarrow u_{n-3}$. Then $D[u_0,u_{n-3},u_{n-2},u_{n-1},u_n]$ is an induced rainbow $CB_5$. It follows that $(u_n,u_{n-3},u_{n-2},u_{n-1},u_0)$ is a rainbow $(v,u)$-path in $D$, a contradiction. So $u_0\rightarrow u_{n-1}$.

Now $(u_0,u_{n-1},u_n)$ is a $(u,v)$-path of length $2$.
\end{proof}

\begin{theorem}\label{z3.5}
Let $D=(X,Y)$ be an $m$-arc-coloured bipartite tournament with $\min\{|X|,|Y|\}\geqslant3$. If all $4$-cycles, $6$-cycles and induced subdigraphs $CB_5$ in $D$ are rainbow, and all induced subdigraphs $TB_4$ in $D$ are properly coloured, then $C_r(D)$ is a KP-digraph.
\end{theorem}

\begin{proof}
According to Theorem \ref{z1.2}, it is sufficient to prove that each cycle of $C_r(D)$ has a symmetrical arc. Suppose to the contrary that there exists a cycle $C$ in $C_r(D)$ containing no symmetrical arc. We will get a contradiction by showing that $C$ has a symmetrical arc. Let $C=(x_0,x_1,\ldots,x_n,x_0)$. Since $C$ has no symmetrical arc, for each $i\in\{0,1,\ldots,n\}$, there exists a rainbow $(x_i,x_{i+1})$-path and no rainbow $(x_{i+1},x_{i})$-path in $D$. The following claim holds directly from Lemma \ref{z3.4}.

\vskip 2mm
\noindent{\bf Claim 1.} For each $i\in\{0,1,\ldots,n\}$, either $(x_i, x_{i+1})\in A(D)$ or there exists a $(x_i,x_{i+1})$-path of length $2$ in $D$.
\vskip 2mm

Let
\[P_{i}=\begin{cases}
(x_{i},x_{i+1}),  & \text  (x_{i},x_{i+1})\in A(D);\\
(x_{i},u_i,x_{i+1}),  & \text  (x_{i},x_{i+1}) \notin A(D).
\end{cases}\]
and $C'=P_0P_1\ldots P_n$. Then $C'$ is a closed walk in $D$.

We consider the following two cases.

\vskip 2mm
\noindent\emph {Case 1.} $n=2$.
\vskip 2mm

Now $C$ is a $3$-cycle. Then not all arcs of $C$ is in $D$ since $D$ is a bipartite tournament. W.l.o.g., assume that $(x_0,x_1)\notin A(D)$. Then $\ell(P_0)=2$, $\ell(P_1)\leqslant 2$ and $\ell(P_2)\leqslant 2$. Now $C'$ is a closed walk with length at most $6$. By Lemma \ref{z3.3} ($b$), $C'$ is a cycle. Since all $4$-cycles and $6$-cycles are rainbow, we have $C'$ is rainbow. Now $P_1P_2$ is a rainbow $(x_1,x_0)$-path and hence $(x_1,x_0)\in A(C_r(D))$. Note that $(x_0,x_1)\in A(C)$, which contradicts $C$ has no symmetrical arc.

\vskip 2mm
\noindent\emph{Case 2.} $n\geqslant3$.
\vskip 2mm

In this case, we set $C'=(v_0,v_1,\ldots,v_k,v_0)$ where $v_0=x_0$ and $k\geqslant n$. By Lemma \ref{z3.3} ($a$), $k$ is odd since $(v_k,v_0)\in A(D)$. Also $v_0,v_3$ are adjacent and $v_0,v_{k-2}$ are adjacent in $D$.

If $(v_3,v_0)\in A(D)$, then $(v_0,v_1,v_2,v_3,v_0)$ is a rainbow $4$-cycle. This implies that $(v_1,v_2,v_3,v_0)$ is a rainbow $(v_1,v_0)$-path and $(v_2,v_3,v_0)$ is a rainbow $(v_2,v_0)$-path. Then $\{(v_1,v_0),(v_2,v_0)\}\subseteq A(C_r(D))$. Note that either $v_1=x_1$ or $v_2=x_1$. We have $(x_1,x_0)\in A(C_r(D))$. Note that  $(x_0,x_1)\in A(C)$, which contradicts $C$ has no symmetrical arc.

If $(v_0,v_{k-2})\in A(D)$, then $(v_0,v_{k-2},v_{k-1},v_{k},v_0)$ is a rainbow $4$-cycle. This implies that $(v_0,v_{k-2},v_{k-1},v_k)$ is a rainbow $(v_0,v_{k})$-path and $(v_0,v_{k-2}, v_{k-1})$ is a rainbow $(v_0,v_{k-1})$-path. Then $\{(v_0,v_{k}),(v_0,v_{k-1})\}\subseteq  A(C_r(D))$. Note that either $v_k=x_n$ or $v_{k-1}=x_n$. We have $(x_0,x_n)\in A(C_r(D))$. Note that $(x_n,x_0)\in A(C)$, which contradicts $C$ has no symmetrical arc.

If $(v_0,v_3)\in A(D)$ and $(v_{k-2},v_0)\in A(D)$, we have $v_3\ne v_{k-2}$ and hence $k-2\geqslant5$. Also there exists $i\in \{1,2,\ldots,\frac{k-5}{2}\}$ such that $(v_0,v_{2i+1})\in A(D)$ and $(v_{2i+3},v_0)\in A(D)$.
Let
\begin{center}
$j_0=\mbox{max}\{i\in \{1,2,\ldots,\frac{k-5}{2}\}\,|\,(v_0,v_{2i+1})\in A(D)$, $(v_{2i+3},v_0)\in A(D)\}$.
\end{center}
Then $(v_0,v_{2j_0+1},v_{2j_0+2},v_{2j_0+3},v_0)$ is a rainbow $4$-cycle.

If $v_{2j_0+1}\in V(C)$, let $v_{2j_0+1}=x_j$. Now $(v_{2j_0+2},v_{2j_0+3},v_0,v_{2j_0+1})$ is a rainbow $(v_{2j_0+2},v_{2j_0+1})$-path and $(v_{2j_0+3},v_0,v_{2j_0+1})$ is a rainbow $(v_{2j_0+3},v_{2j_0+1})$-path. Then $\{(v_{2j_0+2},v_{2j_0+1}),(v_{2j_0+3},v_{2j_0+1})\}\subseteq A(C_r(D))$. Note that either $v_{2j_0+2}=x_{j+1}$ or $v_{2j_0+3}=x_{j+1}$. We have $(x_{j+1},x_j)\in A(C_r(D))$. Note that $(x_j,x_{j+1})\in A(C)$, which contradicts $C$ has no symmetrical arc.

If $v_{2j_0+1}\notin V(C)$, by the definition of $C'$, we have $v_{2j_0},v_{2j_0+2}\in V(C)$. Let $v_{2j_0+2}=x_{j+1}$. By the choice of $j_0$, we have $(v_{2j_0+5},v_0)\in A(D)$. This implies that $(v_0,v_{2j_0+1},v_{2j_0+2},v_{2j_0+3},v_{2j_0+4},$ $v_{2j_0+5},v_0)$ is a rainbow $6$-cycle. So $(v_{2j_0+3},v_{2j_0+4},v_{2j_0+5},v_0,v_{2j_0+1},$ $v_{2j_0+2})$ is a rainbow $(v_{2j_0+3},v_{2j_0+2})$-path and $(v_{2j_0+4},v_{2j_0+5},v_0,v_{2j_0+1},v_{2j_0+2})$ is a rainbow $(v_{2j_0+4},v_{2j_0+2})$-path. Then $\{(v_{2j_0+3},v_{2j_0+2}),(v_{2j_0+4},v_{2j_0+2})\}\subseteq  A(C_r(D))$. Note that either $v_{2j_0+3}=x_{j+2}$ or $v_{2j_0+4}=x_{j+2}$, we have $(x_{j+2},x_{j+1})\in A(C_r(D))$. Note that $(x_{j+1},x_{j+2})\in A(C)$, which contradicts $C$ has no symmetrical arc.

In any case, we get a contradiction. Thus $C_r(D)$ is a \emph{KP}-digraph. \end{proof}

By Observation \ref{z1.4} and Theorem \ref{z3.5}, the following corollary is direct.

\begin{corollary}Let $D=(X,Y)$ be an $m$-arc-coloured bipartite tournament with $\min\{|X|,|Y|\}\geqslant3$. If all $4$-cycles, $6$-cycles and induced subdigraphs $CB_5$ in $D$ are rainbow, and all induced subdigraphs $TB_4$ in $D$ are properly coloured, then $D$ has a RP-kernel.
\end{corollary}

\end{document}